\documentclass[11pt]{article}
\usepackage{a4,amssymb,amsmath,amsthm,amstext,amsfonts}
\usepackage{color}
\usepackage{float}
\usepackage{graphicx}
\usepackage{esint}
\usepackage{a4wide}
\usepackage{babel}
\usepackage{bbm}
\usepackage{mathrsfs}
\usepackage{upgreek}
\usepackage{setspace}
\usepackage{multirow}
\usepackage{placeins}
\usepackage{cleveref}
\usepackage{array}
\usepackage{algorithmic}
\usepackage{algorithm}
\usepackage{bm}
\usepackage{nicefrac}
\usepackage{subfig}
\usepackage{wrapfig}
\numberwithin{equation}{section}
\usepackage[title]{appendix}
\usepackage{authblk}
\newtheorem{mylemma}{Lemma}[section]
\newtheorem{mytheorem}{Theorem}[section]
\newtheorem{myproposition}{Proposition}[section]
\newtheorem{myremark}{Remark}[section]
\newtheorem{mydefinition}{Definition}[section]

\begin{document}
	
	\providecommand{\keywords}[1]{{\noindent \textit{Key words:}} #1}
	\providecommand{\msc}[1]{{\noindent \textit{Mathematics Subject Classification:}} #1}
	
	\title{Global-in-time energy stability: a powerful analysis tool for the gradient flow problem without maximum principle or Lipschitz assumption\\\vskip 0.8cm}

\author[1]{Jingwei Sun}
\author[1] {Haifeng Wang}
\author[1] {Hong Zhang\thanks{Corresponding author. \\E-mail address: \texttt{zhanghnudt@163.com}.}}
\author[2]{Xu Qian}
\author[1]{Songhe Song}
\affil[1]{Department of Mathematics, National University of Defense Technology, Changsha 410073, Hunan Province, China.}
\affil[2]{College of Mathematics Science, University of Wuhan, Wuhan 430072, Hubei Province, China.}
	\date{}
	
	\maketitle

\begin{abstract}
\noindent  Before proving (unconditional) energy stability for gradient flows, most existing studies either require a strong Lipschitz condition regarding the non-linearity or certain $L^{\infty}$ bounds on the numerical solutions (the maximum principle). However, proving energy stability without such premises is a very challenging task. In this paper, we aim to develop a novel analytical tool, namely global-in-time energy stability, to demonstrate energy dissipation without assuming any strong Lipschitz condition or $L^{\infty}$ boundedness. The fourth-order-in-space Swift-Hohenberg equation is used to elucidate the theoretical results in detail. We also propose a temporal second-order accurate scheme for efficiently solving such a strongly stiff equation. Furthermore, we present the corresponding optimal $L^2$ error estimate and provide several numerical simulations to demonstrate the dynamics.
\end{abstract}

\msc{Primary 35K30, 35K55, 65L06, 65T40}\\

\section{Introduction}
\label{intro}

Many physical problems can be modeled by PDEs that take the form of gradient flows, which are often derived from the second law of thermodynamics. It is well-known that a gradient flow is determined by not only the driving free energy but also the dissipation mechanism. Given a free energy functional $E(u)$, denote its variational derivative as $\mu=\delta E/\delta u$. The general form of the gradient flow can be written as 
\begin{equation}\label{eqn:gradientflow}
	\frac{\partial u}{\partial t}= \mathcal{G}\mu,
\end{equation}
equipped with suitable boundary conditions. In the above, a non-positive symmetric operator $\mathcal{G}$ gives the dissipation mechanism, thus the free energy is non-increasing:
\begin{equation}\label{eqn:energydiss_of_gradientflow}
	\frac{\mathrm{d}E}{\mathrm{d}t}=\langle\mu,\mathcal{G}\mu\rangle\le0,
\end{equation} 
where $\langle\phi,\psi\rangle=\int_{\Omega}\phi\cdot\psi\mathrm{d}\mathbf{x}$. Familiar dissipative operators $\mathcal{G}$ include but not limited to $-I,~\Delta,~-(-\Delta)^{\alpha}~(0<\alpha<1)$, corresponding to $L^2,~H^{-1}$ and nonlocal $H^{-\alpha}$ gradient flows, respectively. In general, the free energy functional contains linear and nonlinear terms, which we write explicitly as 
\begin{equation}\label{eqn:energy_of_gradientflow}
	E(u) = \frac{1}{2}\langle u,\mathcal{L}u\rangle + \langle F(u),1\rangle, 
\end{equation}
where $\mathcal{L}$ is a symmetric non-negative linear operator, and $F(u)$ represents the nonlinear potential. 
In particular, if the energy functional associated with \eqref{eqn:gradientflow} is 
\begin{equation}\label{eqn:energyofCH}
	E_{\varepsilon}(u) = \int_{\Omega}\left(\frac{\varepsilon^2}{2}|\nabla u|^2 + F(u)\right)\,\mathrm{d}\mathbf{x},
\end{equation} 
where $\varepsilon$ is a certain positive parameter, the well-known Allen--Cahn (AC, $L^2$ gradient flow) and Cahn--Hilliard (CH, $H^{-1}$ gradient flow) equations could be derived with different nonlinear potentials. Moreover, it follows from \eqref{eqn:energydiss_of_gradientflow} that 
\begin{equation*}
	E_{\varepsilon}(u(t))\leq E_{\varepsilon}(u(s)),\quad\forall t\leq s,
\end{equation*}
which gives a prior control of $H^1$-norm of the solution. 

The energy dissipation property \eqref{eqn:energydiss_of_gradientflow} is always viewed as a key criterion for designing efficient and long-time stable numerical schemes, such as \cite{chen2012linear,cheng2008efficient,sun2023family,LiX2024a,fuzhaohui,lee2019energy}. Among them, Shen and Yang \cite{trancated3} proved energy stability for the AC and CH equations by a nonlinear term truncation technique. More precisely, it is assumed that 
\begin{equation*}
	\max\limits_{u\in\mathbb{R}}|\tilde{f}{'}(u)|\leq \beta,~\tilde{f}{'}(u)~\text{is a suitable modification of the original function derivative}~ f{'}(u),
\end{equation*}
which is what we referred to as the Lipschitz assumption on the non-linearity in the abstract. Also, see the same strategy for example \cite{fengx,trancated1,trancated2,trancated4} and the references therein. 
As Li and Qiao \cite{lidongjsc} mentioned, the main drawback of the aforementioned numerical developments is that to obtain energy stability, one either makes a Lipschitz assumption on the non-linearity such as \cite{trancated3}, or one assumes some additional $L^{\infty}$ bounds on the numerical solution, which is automatically satisfied by AC type equations with the standard maximum principle. It is very desirable to remove these technical obstacles and establish a more reasonable energy stability theory. Therefore, \textit{the purpose of this work is to prove energy stability on applying an efficient scheme to the gradient flow without maximum principle or any Lipschitz assumption}. 

In fact, very few work has been devoted to such analysis. This is partly due to some technical difficulties related to the $H^m$-norm $(m\ge1)$ prior control and $L^{\infty}$ bounds of the solutions, making it difficult to handle nonlinear terms optimally. Notably, by utilizing a log-type interpolation inequality, Li et al. \cite{lidongjsc,lidongsinum} provided rigorous analyses for the CH equation several years ago, employing first-order stabilized semi-implicit (SSI1) and second-order backward differential formula (BDF2) time-stepping schemes. In this paper, however, we shall adopt a different approach by considering another gradient flow, the Swift--Hohenberg (SH, $L^2$ gradient flow) equation, using a second-order accurate, energy-stable exponential-type scheme. Our \textbf{exponential-type} scheme effectively addresses the severe stiffness introduced by the biharmonic linear term and achieves second-order accuracy in a single step. Furthermore, unlike the previous studies \cite{lidongjsc,lidongsinum}, our analytical framework leverages the specific energy structure and employs specialized operator estimations to achieve the desired theoretical results. This framework is also applicable to the phase field crystal (PFC, $H^{-1}$ gradient flow) equation \cite{pfc}, which we leave the discussion in a forthcoming paper.

The SH equation \cite{swift1977} is widely used in the study of phenomena such as Rayleigh-Bénard convection and more elaborate density functional theories of liquid interfaces \cite{hohenberg1992,cross1993,rosa2000,hutt2005,hutt2008}. It differs from the classical AC and CH systems in that the stable phase is periodic, and is built with the following Lyapunov energy functional (cf. \cite{swift1977,dehghan2016,lee2016}):
\begin{equation}\label{1.2}
	E(u) = \int_{\Omega}\frac{1}{2}u\underbrace{(\Delta + 1)^2}_{\mathcal{L}}u + \underbrace{\frac{1}{4}u^4 - \frac{\varepsilon}{2}u^2}_{F(u)}\,\mathrm{d}\mathbf{x},
\end{equation}
which leads to the following fourth-order-in-space system:
\begin{equation}\label{1.1}
	\frac{\partial u}{\partial t}= -(\Delta +1)^2u - f(u),\quad (\mathbf{x}, t)\in\Omega\times (0,T],
\end{equation}
where $f(u) = u^3-\varepsilon u$, $\Omega\subset\mathbb{R}^2$, $u:\Omega\to\mathbb{R}$ stands for the order parameter, and $\varepsilon$ is a positive constant with certain physical significance. Similar to the PFC equation, the SH equation can describe many basic properties of polycrystalline materials that arise during non-equilibrium processing. The equation of motion governing these non-equilibrium phenomena is a nonlinear PDE that generally cannot be solved analytically for random initial conditions. Therefore, numerous efforts have been made to design appropriate numerical methods to help researchers understand and characterize non-equilibrium phenomena, see, e.g., \cite{lee2019energy,yang2021conservative,sujian,yang2021linear,liu2023efficient,LiX2024a} and the references therein.

A significant challenge is that obtaining energy dissipation using explicit schemes is very difficult due to severe time-step restrictions. To ensure the decay of the total energy while employing a moderately large time-step size, a feasible choice is to use implicit-explicit (IMEX) schemes, where the linear part is treated implicitly and the nonlinear part is evaluated explicitly. Motivated by this idea and Duhamel's formula with respect to \eqref{eqn:gradientflow}:
\begin{equation}\label{eqn}
	u(t) = \mathrm{e}^{tG}u(t_0) + \int_{t_0}^t \mathrm{e}^{G(t-s)}f(u(s))\,\mathrm{d}s,
\end{equation}
where $\mathrm{e}^{tG}$ is the standard semi-group (kernel) with $G$ as the spatial discretization of $\mathcal{G}$, researchers have widely used energy-stable exponential-type methods to deal with various gradient flows. For instance, Ju et al. \cite{ju2019,ju2016,li2021} used first- and second-order exponential time differencing Runge-Kutta (ETDRK) methods to present numerical analyses for AC and CH equations, while Chen et al. \cite{chenwenbin1,chenwenbin2} conducted extensive research on the application of stabilization ETD multi-step (sETDMS) schemes for thin film models. Building on these remarkable studies, we shall propose an improved second-order exponential Runge-Kutta method, namely ERK\textit{(2,2)} \eqref{2.6}. Furthermore, \textbf{original} energy stability of the SH equation \eqref{1.1} is rigorously proved by our proposed method, while multi-step methods only achieve modified energy stability, which includes a few correction terms that cannot be removed.

We now state our main results by three steps as follows. \\
\noindent\textbf{\textit{Step 1}}: by assuming $l^{\infty}$ bounds of the numerical solutions, we are able to derive original energy stability of the SH equation:
\begin{mytheorem}[Original energy stability]\label{theorem3.1}
	The ERK(2,2) scheme \eqref{2.6} unconditionally preserves the original energy of the SH equation; that is,
	\begin{equation*}
		E_{N}(u^{n+1})\leq E_{N}(u^n),\quad \forall \tau>0,~0\leq n\leq N_t-1,
	\end{equation*}
	provided that
	\begin{equation}\label{classic_kappa}
		\kappa\ge\max\limits_{|\xi|\le\beta}\frac{|3\xi^2 - \varepsilon|}{2}~\mbox{with}~\beta:=\max_{0\leq n\leq N_t,~i=0,1,2}\|u_{n,i}\|_{\ell^{\infty}}.
	\end{equation}
\end{mytheorem}

To finish the complete proof, it is necessary to recover the $l^{\infty}$ boundedness of $u^{n+1}$. \\
\noindent\textbf{\textit{Step 2}}: We suppose energy stable in the first $n$ steps to obtain an upper bound on the discrete energy, i.e. $E_N(u^n)\leq E_N(u^{n-1})\leq\ldots\leq E_N(u^0)=:C_e$ with $C_e$ a constant. Subsequently, with the help of many state-of-art theoretical analysis techniques, such as the discrete Sobolev inequality, elliptic regularity, as well as repeated eigenvalue estimates for various Fourier-space operators at the RK stages, we are able to establish $\ell^2$ and $H_h^2$ bounds of the numerical solutions at every RK stages. 

\noindent\textbf{\textit{Step 3}}: In turn, $\|u^{n+1}\|_{l^{\infty}}$ becomes a direct consequence of an application of discrete Sobolev embedding. Such an $\ell^\infty$ bound enables us to derive the following desirable result under a certain time-step size constraint:
\begin{mytheorem}[Global-in-time energy stability]\label{theorem3.2}
	With the chosen stabilization parameter in \eqref{classic_kappa}, which only depends on the parameter $\varepsilon$, initial energy $E_0$ and domain $\Omega$, we select a time-step size that satisfies the following $\mathcal{O}(1)$ constraint: 
	\begin{equation*}
		\tau\leq\min\left\{\frac{1}{256},\frac{1}{64}\tilde{C}_1^{-4},(64\kappa)^{-\frac{1}{2}},\frac{1}{4}\tilde{C}_0^{-2}\kappa^{-\frac{1}{2}}\right\},
	\end{equation*} 
	where both $\tilde{C}_1$ and $\tilde{C}_0$ are global-in-time constants. Then the numerical solution $\{u^n\}_{0\leq n\leq N_t}$ produced by the ERK\textit{(2,2)} scheme \eqref{2.6} always satisfies $E_N(u^{n+1})\leq E_N(u^n)$.
\end{mytheorem}

It is observed that with only one condition concerning the time step, the proof is established without additional assumptions. Moreover, the above two theorems complement each other, and the constants $\tilde{C}_1$ and $\tilde{C}_0$ do not depend on exponential factors involving the time parameter. Therefore, even if the total time is large, these factors will not cause instability. This is precisely why we refer to it as ``global-in-time energy stability".

The remainder of the paper is structured as follows. In Section \ref{sec2}, a fully discrete numerical scheme is devised, employing the Fourier spectral collocation spatial discretization and two-stage, second-order exponential Runge--Kutta temporal integration. The detailed proof of main results is provided in \Cref{sec3}, followed by an optimal rate convergence estimate in \Cref{sec4}. In \Cref{sec5}, we present some numerical results to illustrate the temporal accuracy and long-time dynamic performance of the proposed scheme. Moreover, some concluding remarks are made in \Cref{sec6}. The proofs of two essential propositions in \Cref{sec3} are placed in \Cref{secA,secB}.

\section{Fully discretization}
\label{sec2}
To simplify the presentation, a two-dimensional (2-D) domain is assumed, and an extension to the three-dimensional case could be similarly handled without an essential difficulty. Furthermore, we assume throughout the paper that periodic boundary conditions are chosen such that all boundary terms will vanish when integration by parts is performed; of course, an extension to the homogeneous Neumann boundary condition case is straightforward. 

\subsection{Review of the Fourier pseudo-spectral approximation}
We assume that the domain is given by $\Omega = [0,L]^2$, with a uniform mesh size: $N_x = N_y = N$, $N \cdot h= L$. The number of grid points (in each direction) is set as $N=2K+1$, and the case for an even $N$ could be similarly treated. All the spatial variables are evaluated on the $N\times N$ uniform mesh $\Omega_{N}$, in which $x_p=ph,~y_q=qh$, $0\leq p,q\leq N-1$, $h=\frac{L}{N}$. We also denote a uniform time step size $\tau = \frac{T}{N_t}$, where $N_t$ is a positive integer, and $t_n = n\tau$ for $0\leq n\leq N_t$. Let $\mathcal{M}_{N}$ denote the set of 2-D periodic grid functions defined on $\Omega_{N}$. For any gird functions $ {f}, {g}\in\mathcal{M}_{N}$, the discrete $\ell^2$ norm and inner product, and discrete $\ell^{\infty}$ norm are introduced as
\begin{equation}\nonumber
	\Vert {f}\Vert_{2} := \sqrt{\langle f,f\rangle } \quad \text{with} \quad \langle  {f}, {g}\rangle := h^2\sum_{p,q=0}^{N-1}f_{p,q}g_{p,q};\quad \Vert {f}\Vert_{\infty} := \max\limits_{p,q}|f_{p,q}|.
\end{equation} 
For $ {f}\in\mathcal{M}_{N}$, we set the discrete Fourier expansion as
\begin{equation}\label{2.1}
	f_{p,q} = \sum_{\ell,m=-K}^{K}\hat{f}_{\ell,m}\mathrm{e}^{2\pi \mathrm{i}(\ell x_p + my_q)/L} , \quad\text{with}~\mathrm{i}=\sqrt{-1},
\end{equation}
where the coefficients $\hat{f}_{\ell,m}$ are obtained by the discrete Fourier transform.

Although no aliasing error needs to be considered in the numerical analysis, due to the lack of spatial derivative terms in the nonlinear parts of the SH equation \eqref{1.1}, we have to introduce a periodic extension of a grid function and a Fourier collocation interpolation operator to facilitate the later analysis.
\begin{mydefinition}\label{definition2.1}
	For any $f\in\mathcal{M}_N$, we denote its periodic continuous extension into $\mathcal{B}^K$ (the space of trigonometric polynomials of degree at most $K$) as $f_S$, given by
	\begin{equation*}
		f_S(x,y) = S_N(f) = \sum_{\ell,m=-K}^K\hat{f}_{\ell,m}\mathrm{e}^{2\pi\mathrm{i}(\ell x+my)/L}.
	\end{equation*}
	We call $S_N:\mathcal{M}_{N}\to \mathcal{B}^K$ the spectral interpolation operator. 
\end{mydefinition}
\begin{mydefinition}\label{definition2.2}
	The discrete differentiation operators in the $x$-direction are defined as
	\begin{align*}
		&(\mathcal{D}_xf)_{p,q} := \sum_{\ell,m=-K}^{K}\frac{2\ell\pi\mathrm{i}}{L}\hat{f}_{\ell,m}\mathrm{e}^{2\pi\mathrm{i}(\ell x_p+ my_q)/L},\\
		&(\mathcal{D}_x^{2}f)_{p,q} := \sum_{\ell,m=-K}^{K}\frac{-4\ell^2\pi^2}{L^2}\hat{f}_{\ell,m}\mathrm{e}^{2\pi\mathrm{i}(\ell x_p+ my_q)/L}.
	\end{align*}
	The differentiation operators in the $y$-direction, $\mathcal{D}_y$ and $\mathcal{D}_y^{2}$, can be introduced in the same fashion. In turn, the discrete Laplacian, gradient and divergence operators are given by
	\begin{equation}\nonumber
		\Delta_{N} {f} := \mathcal{D}_x^2 {f} + \mathcal{D}_y^2 {f},\quad \nabla_{N} {f}:=\begin{pmatrix}
			\mathcal{D}_x {f},\\ \mathcal{D}_y {f}
		\end{pmatrix},\quad\nabla_{N}\cdot\begin{pmatrix}
			f_1\\f_2
		\end{pmatrix}:= \mathcal{D}_xf_1 + \mathcal{D}_yf_2 , 
	\end{equation}
	at the point-wise level.
\end{mydefinition}

Detailed calculations reveal that the following summation-by-parts formulas (cf. \cite{chen2012linear, chen2014linear, gottlieb2012long, gottlieb2012stability}) are valid .
\begin{myproposition}\label{proposition2.1}
	For any periodic grid functions $ {f}, {g}\in\mathcal{M}_{N}$, we have 
	\begin{displaymath}
		\langle {f},\Delta_{N} {g}\rangle  = -\langle \nabla_{N} {f},\nabla_{N} {g}\rangle  = \langle\Delta_{N} {f}, {g}\rangle ,\quad\langle  {f},\nabla_{N}\cdot {g}\rangle  = -\langle \nabla_{N} {f}, {g}\rangle .
	\end{displaymath}
\end{myproposition}
\subsection{Time-stepping integrator}
The space-discrete problem of equation \eqref{1.1} turns out to be: find $ {u}:[0,+\infty)\to\mathcal{M}_{N}$ which satisfies
\begin{equation}\label{2.2}
	\frac{\mathrm{d}u}{\mathrm{d}t} = -(\Delta_{N} + I)^2u - f(u). 
\end{equation}
Meanwhile, the discrete energy is given by
\begin{equation*}
	E_{N}(u) = \frac{1}{2}\|(\Delta_{N}+I)u\|_2^2 + \langle\frac{1}{4}u^4-\frac{\varepsilon}{2}u^2,1\rangle.
\end{equation*}
Adding and subtracting the stabilization term $\kappa {u}$ ($\kappa>0$) on the right-hand side (RHS) of equation \eqref{2.2}, we obtain the following ODE system: 
\begin{equation}\nonumber
	\left\{\begin{aligned}
		& \frac{\mathrm{d} {u}}{\mathrm{d}t} = -L_{\kappa} {u} + N_{\kappa}(u),\quad t\in(0,T],\\
		&  {u}(0) = u_{0},\qquad \qquad \quad ~~t=0,
	\end{aligned}\right.
\end{equation}
where $u_0\in\mathcal{M}_{N}$ is given by the initial data, $L_{\kappa} = (\Delta_{N}+I)^2 +\kappa I$, and $N_{\kappa}(u) =  \kappa u - f(u)$. Let $\{\mathrm{e}^{-t L_{\kappa}}\}_{t\ge0}$ denote the semigroup on $\Omega_{N}$ with the generator $(-L_{\kappa})$. Based on the integrating factor, an update of the exact solution from time instant $t_n$ to the next time step $t_{n+1}$ could be expressed as
\begin{equation}\label{duhamel}
	{u}(t_{n+1}) = \mathrm{e}^{-\tau L_{\kappa}} {u}(t_n) + \int_{0}^{\tau}\mathrm{e}^{-(\tau-r)L_{\kappa}}N_{\kappa}( {u}(t_n+r))\,\mathrm{d}r.
\end{equation}
Denoting by $u^n$ the fully discrete numerical solution at the time step $t_n$, a family of second-order ERK approach derived by Hochbruck and Ostermann \cite{hochbruck2005} for solving equation \eqref{duhamel} is formulated as below: for $n=0,1,\ldots,N_t-1$,
\begin{equation}\label{general_erk}
	u_{n,i} = \varphi_0	(\tau L_{\kappa})u^n + \tau\sum_{j=0}^{i-1}a_{i,j}(\tau L_{\kappa})N_{\kappa}(u_{n,j}),\quad i = 1,2,
\end{equation}
in which the coefficients $a_{i,j}(z)$ that satisfy the second-order conditions in Section 5.1 of \cite{hochbruck2005} are linear combinations of the $\varphi_k(z)$ functions defined by
\begin{equation*}
	\varphi_0(z) = \mathrm{e}^{-z},\quad \varphi_{1}(z) = \frac{1-\mathrm{e}^{-z}}{z},\quad\varphi_2(z) = \frac{\mathrm{e}^{-z}-1+z}{z^2},~\text{with}~\varphi_k(0) = \frac{1}{k!}.
\end{equation*}

It is observed that $a_{i,j}(z)$ are specified in the following one-parameter Butcher-like tableau~(cf. equation (5.8) in \cite{hochbruck2005}):
\begin{align}\label{2.5}
	\begin{array}
		{>{\centering\arraybackslash$} p{0.4cm} <{$} | >{\centering\arraybackslash$} p{1.2cm} <{$} >{\centering\arraybackslash$} p{0.8cm} <{$}}
		c_0 & 0 & \\
		c_1 &  a_{1,0}(z) & 0\\
		\hline
		c_2 & a_{2,0}(z) & a_{2,1}(z)	
	\end{array} ~~=~~ 
	\begin{array}
		{>{\centering\arraybackslash$} p{0.4cm} <{$} | >{\centering\arraybackslash$} p{1.8cm} <{$} >{\centering\arraybackslash$} p{0.8cm} <{$}}
		0 & 0 & \\
		c_1 &  c_1\varphi_{1,1} & 0\\
		\hline
		1 & (1-\frac{1}{2 c_1})\varphi_1 & \frac{1}{2 c_1}\varphi_1	
	\end{array},
\end{align}
where $\varphi_{1,1}(z):=\varphi_{1}( c_1z)$. Notice that the choice $ c_1=\frac{1}{2}$ yields $a_{2,0} = 0$. We apply this particular instance to derive a two-stage, second-order fully discrete scheme, denoted as ERK\textit{(2,2)}:
\begin{equation}\label{2.6}
	\begin{aligned}
		{u}^{n+1} &= \varphi_0(\tau L_{\kappa}) {u}^n + \tau\varphi_{1}(\tau L_{\kappa})N_{\kappa}(u_{n,1})\\
		&= P^{-1}\left[\varphi_0(\tau\widehat{L}_{\kappa})Pu^n + \tau\varphi_{1}(\tau \widehat{L}_{\kappa})PN_{\kappa}(u_{n,1})\right],\\
		u_{n,1} &= \varphi_0(\frac{1}{2}\tau L_{\kappa}) {u}^n + \frac{\tau}{2}\varphi_{1}(\frac{1}{2}\tau L_{\kappa})N_{\kappa}( {u}^n)\\
		&=P^{-1}\left[\varphi_0(\frac{1}{2}\tau \widehat{L}_{\kappa})Pu^n + \frac{\tau}{2}\varphi_{1}(\frac{1}{2}\tau\widehat{L}_{\kappa})PN_{\kappa}( {u}^n)\right],
	\end{aligned}
\end{equation}
where the operators $P$ and $P^{-1}$ can be implemented by the 2-D fast Fourier transform (FFT) and the corresponding inverse transform, respectively. Therefore, the overall computational complexity is $\mathcal{O}(N^2\log_2N)$ per time step. In addition, $\widehat{L}_{\kappa} = PL_{\kappa}P^{-1}$.
\begin{myremark}\label{remark2.1}
	Note that another second-order typical ERK scheme, ETDRK2 \eqref{eqn:etdrk2},
	shares the same stage as ERK\textit{(2,2)} but possesses a more complex structure. In a previous work \cite{sun2023family}, the authors conducted a numerical investigation revealing that ETDRK2 achieves better accuracy than the scheme \eqref{2.5} with $c_1=1$. Nonetheless, when compared with ETDRK2, ERK\textit{(2,2)} demonstrates both enhanced accuracy and faster computation (cf. \Cref{fig_err} in Section \ref{sec5}), representing a significant computational advantage.
\end{myremark}
\begin{myremark}
	In fact, there is one more difference between the two methods; that is, ETDRK2 is of \textit{stiff order two}, while ERK\textit{(2,2)} has \textit{stiff convergence order two} \cite{Maset2008}. Existing numerical evidence suggests that the latter method demonstrates superior error accuracy due to its second-order global error, whereas the former exhibits enhanced stability when addressing stiff problems. Remaining this observation in actual numerical implementation remains to be studied further.
\end{myremark}
\section{Proof of Main results}
\label{sec3}
Now we turn to the energy stability analysis of the ERK\textit{(2,2)} scheme. For any $ {u}\in\mathcal{M}_{N}$,  the discrete energy functional could be rewritten as $E_{N}(u) = E_{N,c}(u) + E_{N,e}(u)$, with
\begin{equation}\nonumber
	\nonumber E_{N,c}({u}) = \frac{1}{2}\langle {u}, (\Delta_{N}+I)^2 {u} \rangle , \quad  E_{N,e}( {u})= \langle\frac{1}{4} {u}^4 - \frac{\varepsilon}{2} {u}^2, {1}\rangle.
\end{equation}
\subsection{Proof of \Cref{theorem3.1}}
We first prove that the ERK\textit{(2,2)} scheme \eqref{2.6}, with a sufficiently large $\kappa$, is original energy stable.
\noindent
\begin{proof}
	A difference of the two energy functional $E_{N,c}( {v})$ and $E_{N,c}( {w})$ yields
	\begin{equation}\label{3.43}
		\begin{aligned}
			&E_{N,c}( {v}) - E_{N,c}( {w}) = \frac{1}{2}\langle  {v},(\Delta_{N} + I)^2 {v}\rangle  - \frac{1}{2}\langle  {w},(\Delta_{N} + I)^2 {w}\rangle \\
			&= \langle  {v}- {w},(\Delta_{N}+I)^2 {v}\rangle  - \frac{1}{2}\langle  {v}- {w},(\Delta_{N}+I)^2( {v}- {w})\rangle \\
			&\leq \langle  {v}- {w},L_{\kappa} {v}\rangle  - \kappa\langle {v}- {w}, {v}\rangle .\\
		\end{aligned}
	\end{equation}
	For $E_{N,e}( {v})$ and $E_{N,e}( {w})$, a careful application of Taylor's expansion indicates that 
	\begin{equation}\label{3.44}
		\begin{aligned}
			& E_{N,e}( {v}) - E_{N,e}( {w}) = \langle  {v}- {w}, f( {w})\rangle  + \frac{1}{2}\langle {v}- {w}, f{'}( {\xi})( {v}- {w})\rangle \\
			& = -\langle  {v}- {w}, N_{\kappa}( {w}) \rangle  + \kappa\langle  {v}- {w}, {w}\rangle  + \frac{1}{2}\langle {v}- {w}, f{'}( {\xi})( {v}- {w})\rangle ,
		\end{aligned}
	\end{equation}
	in which the variable $ {\xi}$ is between $ {v}$ and $ {w}$, at a point-wise level. 
	
	\noindent A combination of \eqref{3.43} and \eqref{3.44} leads to
	\begin{equation}\nonumber
		\begin{aligned}
			&E_{N}( {v}) - E_{N}( {w}) \leq\langle  {v}- {w},L_{\kappa} {v} - \kappa {v}\rangle  + \langle  {v}- {w}, -N_{\kappa}( {w})+\kappa {w}\rangle 
			+ \frac{1}{2}\langle {v}- {w}, f{'}( {\xi})( {v}- {w})\rangle \\
			&= \langle  {v}- {w}, L_{\kappa} {v} - N_{\kappa}( {w}) \rangle  + \langle {v}- {w},[\frac{1}{2}f{'}( {\xi})-\kappa I]( {v}- {w})\rangle .
		\end{aligned}
	\end{equation}
	In turn, under the condition that $\kappa\ge\max\limits_{|\xi|\le\beta}\frac{|3\xi^2-\varepsilon|}{2}\ge\frac{1}{2}f'(\xi)$, the following inequality is available:
	\begin{equation*}
		E_{N}( {v}) - E_{N}( {w}) \leq\langle  {v}- {w}, L_{\kappa} {v} - N_{\kappa}( {w}) \rangle  .
	\end{equation*}
	Then we arrive at 
	\begin{equation}\label{3.45}
		\begin{aligned}
			&E_{N}(u_{n,1})- E_{N}( {u}^n) \leq\langle u_{n,1} -  {u}^n, L_{\kappa}u_{n,1} - N_{\kappa}( {u}^n)\rangle \\
			&=\langle u_{n,1} -  {u}^n, L_{\kappa}u_{n,1} - [c_1\tau\varphi_{1}(c_1\tau L_{\kappa})]^{-1}[u_{n,1} - u^n + c_1\tau L_{\kappa}\varphi_{1}(c_1\tau L_{\kappa}) {u}^n]\rangle \\
			&=\langle u_{n,1} -  {u}^n, \{L_{\kappa} - \left[c_1\tau\varphi_{1}\left(c_1\tau L_{\kappa}\right)\right]^{-1}\}(u_{n,1}- {u}^n)\rangle \\
			&=:\langle u_{n,1} -  {u}^n, \Delta_1(u_{n,1}- {u}^n)\rangle ,
		\end{aligned}
	\end{equation}
	where $\Delta_1=2\tau^{-1}h_1(c_1\tau L_{\kappa})$, with $h_1(z) = z - [\varphi_{1}(z)]^{-1}={z\mathrm{e}^{-z}}/{(\mathrm{e}^{-z}-1)}$. Notice that $L_{\kappa}$ is symmetric positive definite, thus $h_1(z) <0$ for any $z\neq0$. Consequently, the operator $\Delta_1$ is negative semi-definite and $E_{N}(u_{n,1}) - E_{N}( {u}^n)\leq 0$. 
	
	\noindent As for the second step, a difference between $E_{N}( {u}^{n+1})$ and $E_{N}(u_{n,1})$ reveals that 
	\begin{equation}\label{3.46}
		\begin{aligned}
			&E_{N}( {u}^{n+1}) - E_{N}(u_{n,1})\leq\langle {u}^{n+1} - u_{n,1}, L_{\kappa} {u}^{n+1} - N_{\kappa}(u_{n,1})\rangle \\
			&=\langle {u}^{n+1} - u_{n,1}, L_{\kappa} {u}^{n+1} -\left[\tau\varphi_{1}(\tau L_{\kappa})\right]^{-1}[{u}^{n+1} - {u}^{n} + \tau L_{\kappa}\varphi_{1}(\tau L_{\kappa}) {u}^n] \rangle \\
			&=\langle {u}^{n+1} - u_{n,1}, \{L_{\kappa} - [\tau\varphi_{1}(\tau L_{\kappa})]^{-1}\}( {u}^{n+1}- {u}^n) \rangle \\
			&=:\langle {u}^{n+1} - u_{n,1}, \Delta_2( {u}^{n+1}- {u}^n) \rangle ,
		\end{aligned}
	\end{equation}
	with $\Delta_2 = \tau^{-1}h_1(\tau L_{\kappa})$ a negative semi-definite operator. In turn, we obtain $E_{N}({u}^{n+1}) - E_{N}(u_{n,1})\leq 0$. Subsequently, a summation of inequalities \eqref{3.45} and \eqref{3.46} yields 
	\begin{equation}\label{3.47}
		\begin{aligned}
			&E_{N}( {u}^{n+1}) - E_{N}( {u}^n)\leq\langle u_{n,1} -  {u}^n, \Delta_1(u_{n,1}- {u}^n)\rangle  + \langle {u}^{n+1} - u_{n,1}, \Delta_2( {u}^{n+1}- {u}^n) \rangle \\
			&=\underbrace{\langle u_{n,1} -  {u}^n, \Delta_1(u_{n,1}- {u}^n)\rangle }_{A} + \underbrace{\langle {u}^{n+1} - u_{n,1}, \Delta_2( {u}^{n+1}-u_{n,1})\rangle }_{B} \\
			&+ \underbrace{\langle  {u}^{n+1} - u_{n,1},\Delta_2(u_{n,1} -  {u}^n) \rangle }_{C},\\
		\end{aligned}
	\end{equation}
	%where
	\begin{equation*}
		A = \underbrace{\langle u_{n,1} -  {u}^n, (\Delta_1-\frac{1}{2}\Delta_2)(u_{n,1}- {u}^n)\rangle }_{A_1} + \underbrace{\frac{1}{2}\langle u_{n,1} -  {u}^n, \Delta_2(u_{n,1}- {u}^n)\rangle }_{A_2},
	\end{equation*}
	where $\Delta_1-\frac{1}{2}\Delta_2 = 2\tau^{-1}h_2(\tau L_{\kappa})$, with $h_2(z) = h_1(\frac{z}{2}) - \frac{1}{4}h_1(z) = \frac{z\mathrm{e}^{-z/2}}{2\left(\mathrm{e}^{-z/2}-1\right)} - \frac{z\mathrm{e}^{-z}}{4\left(\mathrm{e}^{-z}-1\right)}$. It can be verified that $h_2\leq 0$ for any $z\neq 0$. Therefore, $\Delta_1-\frac{1}{2}\Delta_2$ is symmetric negative semi-definite and $A_1\leq 0$. We also notice that
	\begin{equation}\nonumber
		\begin{aligned}
			A_2 + \frac{1}{2}C  = \frac{1}{2}\langle {u}^{n+1}- {u}^n, \Delta_2(u_{n,1}- {u}^n)\rangle ,\quad\frac{1}{2}(B+C) = \frac{1}{2}\langle  {u}^{n+1}-u_{n,1},\Delta_2( {u}^{n+1}- {u}^n) \rangle ,
		\end{aligned}
	\end{equation}
	which in turn leads to 
	\begin{equation}\nonumber
		A_2 + \frac{1}{2}B + C = \frac{1}{2}\langle {u}^{n+1}- {u}^n,\Delta_2( {u}^{n+1}- {u}^n)\rangle \leq 0.
	\end{equation}
	As a result, inequality \eqref{3.47} turns out to be
	\begin{equation}\nonumber
		E_{N}( {u}^{n+1}) - E_{N}( {u}^n)\leq A_1 + (A_2 + \frac{1}{2}B + C) + \frac{1}{2}B\leq 0.
	\end{equation}
	This completes the proof.
\end{proof}
\subsection{Proof of \Cref{theorem3.2}}
\Cref{theorem3.1} indicates that, the choice of stabilization parameter in \eqref{classic_kappa} implicitly uses an \textit{a-priori} $\ell^{\infty}$ bound assumption on $u_{n,i}$ in order to make $\kappa$ a controllable constant. Of course, it is desirable to remove this technical restriction and establish a more reasonable energy stability theory.

In this subsection, we perform a direct analysis for the numerical solution of the ERK\textit{(2,2)} scheme \eqref{2.6}, so that uniform-in-time $\ell^2$ and $H_h^2$ estimates become available for the numerical solutions at all the stages. With the help of the discrete Sobolev embedding, we are able to recover the global-in-time values for $\beta$ and $\kappa$ in \eqref{classic_kappa}, which allows us to derive a global-in-time energy stability estimate for the ERK\textit{(2,2)} scheme.
\subsubsection{A few preliminary notations and results}
\noindent

The verification of the following Calculus-style analysis is straightforward. 

\begin{mylemma}\label{lemma3.1}
	For the ERK\textit{(2,2)} scheme \eqref{2.6}, the functions $\varphi_i(c_1z)$ are decreasing and $\varphi_i(c_1z)\leq1$, $\forall z>0$,
	where $i=0,1$.
\end{mylemma}

We introduce the following linear operators to facilitate the energy stability analysis:
\begin{equation}\nonumber
	G_i = c_i\varphi_{1}(c_i\tau L_{\kappa}) = (\tau L_{\kappa})^{-1}(I-\mathrm{e}^{-c_i\tau L_{\kappa}}),\quad i=1, 2.
\end{equation}
In more details, for ${f}\in\mathcal{M}_{N}$ with a discrete Fourier expansion as~\eqref{2.1}, an application of of $G_i$ becomes
\begin{equation}\nonumber
	\left(G_if\right)_{p,q} = \sum_{\ell,m=-K}^{K}c_i\varphi_{1}(c_i\tau\Lambda_{\ell,m})\hat{f}_{\ell,m}\mathrm{e}^{2\pi\mathrm{i}(\ell x_p+my_q)/L},~\text{with}~\Lambda_{\ell,m} = (-1+\lambda_{\ell,m})^2+\kappa,
\end{equation}
where $\lambda_{\ell,m}=\frac{4\pi^2}{L^2}(\ell^2 + m^2)$. Since all the eigenvalues of $G_i$ are non-negative, we define $G_i^{\frac{1}{2}}$ and $G_{1,2}^{\frac{1}{2}} = G_1^{\frac{1}{2}}G_2^{\frac{1}{2}}$ as 
\begin{equation}\nonumber
	\begin{aligned}
		&(G_i^{\frac{1}{2}}f)_{p,q} = \sum_{\ell,m=-K}^K\left(c_i\varphi_{1}(c_i\tau\Lambda_{\ell,m})\right)^{\frac{1}{2}}\hat{f}_{\ell,m}\mathrm{e}^{2\pi\mathrm{i}(\ell x_p+my_q)/L};\\
		&(G_{1,2}^{\frac{1}{2}}f)_{p,q} = \sum_{\ell,m=-K}^K\left(c_1c_2\varphi_{1}(c_1\tau\Lambda_{\ell,m})\varphi_{1}(c_2\tau\Lambda_{\ell,m})\right)^{\frac{1}{2}}\hat{f}_{\ell,m}\mathrm{e}^{2\pi\mathrm{i}(\ell x_p+my_q)/L}.
	\end{aligned}
\end{equation}
It is obvious that the operators $G_i^{\frac{1}{2}}$ and $G_{1,2}^{\frac{1}{2}}$ are commutative with any differential operators in the Fourier pseudo-spectral space, and the summation by parts formulas are available:
\begin{equation}\nonumber
	\langle {f}, G_i {g}\rangle  = \langle G_i^{\frac{1}{2}} {f}, G_i^{\frac{1}{2}} {g}\rangle,\quad \langle G_1{f}, G_2{g}\rangle  = \langle G_{1,2}^{\frac{1}{2}} {f}, G_{1,2}^{\frac{1}{2}} {g}\rangle.
\end{equation}

In addition, the operators $G_i^{*}$, $G_i^{**}$, $G^{*}$, and $G^{**}$ are introduced to facilitate the analysis for the diffusion part:
\begin{equation}\nonumber
	\begin{aligned}
		&\left(G_i^{*}f\right)_{p,q} = \sum_{\ell,m=-K}^{K}\left(c_i\Lambda_{\ell,m}\varphi_{1}(c_i\tau\Lambda_{\ell,m})\right)^{\frac{1}{2}}\hat{f}_{\ell,m}\mathrm{e}^{2\pi\mathrm{i}(\ell x_p+my_q)/L};\\
		&\left(G_i^{**}f\right)_{p,q} = \sum_{\ell,m=-K}^{K}\left(c_i\Lambda_{\ell,m}\varphi_{1}(c_i\tau\Lambda_{\ell,m})\right)^{\frac{1}{2}}\lambda_{\ell,m}\hat{f}_{\ell,m}\mathrm{e}^{2\pi\mathrm{i}(\ell x_p+my_q)/L};\\
		&\left(G^{*}f\right)_{p,q} = \sum_{\ell,m=-K}^{K}\left(c_1c_2\Lambda_{\ell,m}^2\varphi_{1}(c_1\tau\Lambda_{\ell,m})\varphi_{1}(c_2\tau\Lambda_{\ell,m})\right)^{\frac{1}{2}}\hat{f}_{\ell,m}\mathrm{e}^{2\pi\mathrm{i}(\ell x_p+my_q)/L};\\
		&\left(G^{**}f\right)_{p,q} = \sum_{\ell,m=-K}^{K}\left(c_1c_2\Lambda_{\ell,m}^2\varphi_{1}(c_1\tau\Lambda_{\ell,m})\varphi_{1}(c_2\tau\Lambda_{\ell,m})\right)^{\frac{1}{2}}\lambda_{\ell,m}\hat{f}_{\ell,m}\mathrm{e}^{2\pi\mathrm{i}(\ell x_p+my_q)/L}.
	\end{aligned}
\end{equation}
The following identities could be verified in a straightforward way:
\begin{equation}\nonumber
	\langle G_iL_{\kappa} {f}, (I+\Delta_{N}^2){f}\rangle  = \Vert G_i^* {f}\Vert_{2}^2 + \Vert G_i^{**}f\Vert_{2}^2;~
	\langle G_1L_{\kappa} {f}, G_2(I+\Delta_N^2) {f}\rangle  = \Vert G^{*} {f}\Vert_{2}^2 + \Vert G^{**}f\Vert_{2}^2.
\end{equation}
By the fact that $\varphi_1(c_1z)\leq 1$ (by \Cref{lemma3.1}), it is easy to derive \Cref{proposition3.1}. We also introduce other two propositions whose proofs are postponed to \Cref{secA,secB}.
\begin{myproposition}\label{proposition3.1}
	For any ${f}\in\mathcal{M}_{N}$, the following two estimates are valid: 
	\begin{equation}\label{eqn:gi}
		\Vert G_i^{\frac{1}{2}}f\Vert_{2}\leq\Vert f\Vert_{2},\qquad\Vert G_{1,2}^{\frac{1}{2}}f\Vert_{2}\leq\Vert f\Vert_{2}.
	\end{equation}
\end{myproposition}
\begin{myproposition}\label{proposition3.2}
	For $\kappa\ge1$ and $f\in\mathcal{M}_{N}$, we have
	\begin{equation}\label{eqn:proposition3.2(1)}
		\begin{aligned}
			& \Vert G_i^{*} {f}\Vert_{2}^2 + \Vert G_i^{**} {f}\Vert_{2}^2\ge\frac{1}{4}(\Vert G_i^{\frac{1}{2}}\Delta_{N}{f}
			\Vert_{2}^2+\Vert G_i^{\frac{1}{2}}\Delta_{N}^2{f}
			\Vert_{2}^2) \\
			& + (\kappa-1)(\Vert G_i^{\frac{1}{2}}{f}
			\Vert_{2}^2+\Vert G_i^{\frac{1}{2}}\Delta_{N}{f}
			\Vert_{2}^2) + \frac{2}{3}(\Vert G_i^{\frac{1}{2}}{f}
			\Vert_{2}^2+\Vert G_i^{\frac{1}{2}}\Delta_N{f}
			\Vert_{2}^2).
		\end{aligned}
	\end{equation}
	\begin{equation}\label{eqn:proposition3.2(2)}
		\begin{aligned}
			&\Vert G^{*} {f}\Vert_{2}^2 + \Vert G^{**} {f}\Vert_{2}^2\ge\frac{1}{4}(\Vert G_{1,2}^{\frac{1}{2}}\Delta_{N}{f}
			\Vert_{2}^2 + \Vert G_{1,2}^{\frac{1}{2}}\Delta_{N}^2{f}
			\Vert_{2}^2 ) \\
			& + (\kappa-1)(\Vert G_{1,2}^{\frac{1}{2}}{f}
			\Vert_{2}^2 + \Vert G_{1,2}^{\frac{1}{2}}\Delta_{N}{f}
			\Vert_{2}^2) + \frac{2}{3}(\|G_{1,2}^{\frac{1}{2}}f\|_2^2 + \|G_{1,2}^{\frac{1}{2}}\Delta_Nf\|_2^2).
		\end{aligned}
	\end{equation}
\end{myproposition}
\begin{myproposition}\label{proposition3.3}
	For $f,g\in\mathcal{M}_{N}$, we have 
	\begin{align}
		&\tau\langle G_iL_{\kappa} {f},\varphi_0(c_i\tau L_{\kappa})f\rangle  + \Vert {g}-\varphi_0(c_i\tau L_{\kappa})f\Vert_{2}^2\geq\tau\Vert G_i^{*} {g}\Vert_{2}^2;\label{eqn:proposition3.3(1)}\\
		&\tau\langle G_iL_{\kappa} {f},\Delta_{N}^2\varphi_0(c_i\tau L_{\kappa}) {f}\rangle  + \Vert\Delta_{N}( {g}-\varphi_0(c_i\tau L_{\kappa}){f})\Vert_{2}^2\geq\tau\Vert G_i^{**} {g}\Vert_{2}^2;\label{eqn:proposition3.3(2)}\\
		&\tau\langle G_1L_{\kappa} {f},G_2\varphi_0(c_i\tau L_{\kappa}) {f}\rangle  + \Vert G_2^{\frac{1}{2}}(g-\varphi_0(c_i\tau L_{\kappa})f)\Vert_{2}^2\geq\tau\Vert G^{*} {g}\Vert_{2}^2;\label{eqn:proposition3.3(3)}\\
		&\tau\langle G_1L_{\kappa} {f},G_2\Delta_{N}^2\varphi_0(c_i\tau L_{\kappa}) {f}\rangle  + \Vert G_2^{\frac{1}{2}}\Delta_{N}( {g}-\varphi_0(c_i\tau L_{\kappa}){f})\Vert_{2}^2\geq\tau\Vert G^{**} {g}\Vert_{2}^2 . \label{eqn:proposition3.3(4)} 
	\end{align}
\end{myproposition}
\begin{mylemma}\label{lemma3.2}
	For $f\in\mathcal{M}_{N}$, %let $\mathcal{I}_Nf$ be its continuous extension into $\mathcal{B}^K$, 
	it follows that
	\begin{displaymath}
		\Vert {f}\Vert_{\infty}\leq\hat{C}\left(\Vert {f}\Vert_{2} + \Vert \Delta_{N} {f}\Vert_{2}\right),
	\end{displaymath}
	where the constant $\hat{C}$ is independent of $ {f}$, $h$, and $\kappa$.
\end{mylemma}
\begin{proof}
	For any periodic function $f\in\mathcal{M}_{N}$, we recall its continuous extension, $f_S = S_N(f)\in\mathcal{B}^K$, as introduced in \Cref{definition2.1}. Since $f$ is the point-wise interpolation of $f_S$, we see that $\|f\|_{\infty}\le\|f_S\|_{L^{\infty}}$. For any smooth function $f_S$, applying the 2-D Sobolev inequality associated the embedding $H^2\hookrightarrow L^{\infty}$ and the elliptic regularity, it holds that
	\begin{equation*}
		\|f_S\|_{L^{\infty}}\le\hat{C}(\|f_S\|_{L^2} + \|\Delta f_S\|_{L^2}) = \hat{C}(\|f\|_{2} + \|\Delta_Nf\|_{2}) .
	\end{equation*}
\end{proof}

To proceed the energy analysis, we make an \textit{a-priori} assumption at the previous time step: 
\begin{equation}\label{eqn:ass_energy}
	E_{N}( {u}^n)\leq E_{N}( {u}^0)=:C_e.
\end{equation}
Such an assumption will be recovered at the next time step. Afterwards, the $\ell^2$  and $H_h^2$ bounds for the numerical solutions $u^n$ could be derived (cf. Lemma 3.7 of \cite{wise2009energy}),
\begin{equation}\label{eqn:wise2009}
	\Vert {u}^n\Vert_{2},~\Vert\Delta_{N} {u}^n\Vert_{2}\leq C_0:= 2\sqrt{C_e+|\Omega|}.
\end{equation}
Based on an application of \Cref{lemma3.2}, we have a discrete $\ell^{\infty}$  bound at time step $t^n$:
\begin{equation}\label{eqn:infinity_un}
	\|u^n\|_{\infty}\leq\hat{C}(\|u^n\|_2 + \|\Delta_{N}u^n\|_2)\leq 2\hat{C}C_0=:\tilde{C}_0.
\end{equation}
\subsubsection{Preliminary $\ell^2$ and $H_h^2$ estimates of $u_{n,1}$}
To obtain a rough $\ell^{\infty}$ estimate for the intermediate-stage numerical solution, we have to derive the $\ell^2$ and $H_h^2$ estimates for $u_{n,1}$, given by \eqref{2.6}. Meanwhile, an intuitional interaction between the linear and nonlinear terms is not clearly presented in the current numerical formulation \eqref{2.6}. To remedy this issue, by introducing $ u_{n,i}^*=\varphi_0\left(c_i\tau L_{\kappa}\right) {u}^n$, the algorithm can be recast as a two-substage system, facilitating a more convenient theoretical analysis: 
\begin{align}
	\frac{{u}_{n,i}^{*}- {u}^n}{\tau}&= -c_iL_{\kappa}\varphi_{1}(c_i\tau L_{\kappa}) {u}^n,\label{eqn:substage_1}\\
	\frac{u_{n,i} - {u}_{n,i}^{*}}{\tau}& = c_i\varphi_{1}(c_i\tau L_{\kappa})N_{\kappa}(u_{n,j}),\quad i=1,2,~j=i-1.\label{eqn:substage_2}
\end{align}
Taking a discrete $\ell^2$ inner product with equality \eqref{eqn:substage_1} by $(I + \Delta_{N}^2)(u_{n,i}^{*}+ u^n)$, combined with the summation by parts formula, yields
\begin{equation}\label{3.2.2(1)}
	(I+\Delta_{N}^2)(\Vert {u}_{n,i}^{*}\Vert_{2}^2 - \Vert {u}^n\Vert_{2}^2) + \tau(\Vert G_i^{*} {u}^n\Vert_{2}^2 + \Vert G_i^{**} {u}^n\Vert_{2}^2 + \langle G_iL_{\kappa} {u}^n, (I+\Delta_{N}^2){u}_{n,i}^{*}\rangle)  =  0.
\end{equation}
Taking a discrete $\ell^2$ inner product with equality \eqref{eqn:substage_2} by $2(I+\Delta_{N}^2)u_{n,i}$ leads to
\begin{equation}\label{3.2.2(2)}
	\langle u_{n,i} - {u}_{n,i}^{*},2(I+\Delta_{N}^2)u_{n,i}\rangle  = 2\tau\langle G_iN_{\kappa}(u_{n,j}),(I+\Delta_{N}^2)u_{n,i}\rangle .
\end{equation}
The term on the left-hand side (LHS) of equality \eqref{3.2.2(2)} can be rewritten as: 
\begin{equation}\label{3.2.2(3)}
	\begin{aligned}
		&\langle u_{n,i} - {u}_{n,i}^{*},2(I+\Delta_{N}^2)u_{n,i}\rangle = \Vert u_{n,i}\Vert_{2}^2 - \Vert u_{n,i}^{*} \Vert_{2}^2 + \Vert u_{n,i} - {u}_{n,i}^{*}\Vert_{2}^2 + \Vert \Delta_{N}u_{n,i}\Vert_{2}^2\\
		& - \Vert\Delta_{N}u_{n,i}^{*} \Vert_{2}^2 + \Vert\Delta_{N}(u_{n,i} - {u}_{n,i}^{*})\Vert_{2}^2 ,
	\end{aligned}
\end{equation}
where the identity
$a^2 - b^2 = 2a(a-b) - (a-b)^2$ has been employed. In turn, a combination of equalities \eqref{3.2.2(1)}-\eqref{3.2.2(3)} leads to 
\begin{equation}\label{3.2.2(4)}
	\begin{aligned}
		&\Vert u_{n,i}\Vert_{2}^2 - \Vert u^{n} \Vert_{2}^2 + \Vert u_{n,i} - {u}_{n,i}^{*}\Vert_{2}^2 + \Vert \Delta_{N}u_{n,i}\Vert_{2}^2 - \Vert\Delta_{N}u_{n,i}^{*} \Vert_{2}^2 + \Vert\Delta_{N}(u_{n,i} - {u}_{n,i}^{*})\Vert_{2}^2 \\
		& + \tau(\Vert G_i^{*} {u}^n\Vert_{2}^2  + \Vert G_i^{**} {u}^n\Vert_{2}^2 + \langle G_iL_{\kappa} {u}^n, (I+\Delta_{N}^2){u}_{n,i}^{*}\rangle) = 2\tau\langle G_iN_{\kappa}(u_{n,j}),(I+\Delta_{N}^2)u_{n,i}\rangle.
	\end{aligned}
\end{equation}
Meanwhile, by inequalities \eqref{eqn:proposition3.3(1)}-\eqref{eqn:proposition3.3(2)} in \Cref{proposition3.3}, we see that
\begin{equation}\nonumber
	\begin{aligned}
		&\Vert u_{n,i} - u_{n,i}^{*} \Vert_{2}^2 +\Vert\Delta_{N}(u_{n,i} - u_{n,i}^{*})\Vert_{2}^2 + \tau\langle G_iL_{\kappa} {u}^n, (I+\Delta_{N}^2)u_{n,i}^{*}\rangle\\
		&\geq\tau(\Vert G_i^{*}u_{n,i}\Vert_{2}^2
		+\Vert G_i^{**}u_{n,i}\Vert_{2}^2).
	\end{aligned}
\end{equation}
Going back to equality \eqref{3.2.2(4)}, we arrive at
\begin{equation}\label{3.2.2(5)}
	\begin{aligned}
		&\Vert u_{n,i}\Vert_{2}^2 - \Vert u^n\Vert_{2}^2 + \Vert\Delta_{N}u_{n,i}\Vert_{2}^2 - \Vert\Delta_{N}u^n\Vert_{2}^2+ \tau(\Vert G_i^{*}u_{n,i}\Vert_{2}^2 + \Vert G_i^{*} {u}^n\Vert_{2}^2 \\
		&+ \Vert G_i^{**}u_{n,i}\Vert_{2}^2 + \Vert G_i^{**} {u}^n\Vert_{2}^2)
		\leq 2\tau\langle G_iN_{\kappa}(u_{n,j}),(I+\Delta_{N}^2)u_{n,i}\rangle.
	\end{aligned}
\end{equation}
The RHS of inequality \eqref{3.2.2(5)} contains two parts: 
\begin{equation}\label{3.2.2(6)}
	\begin{aligned}
		2\langle G_iN_{\kappa}(u_{n,j}),(I+\Delta_{N}^2)u_{n,i} \rangle  = \langle -2G_iu^3_{n,j} + 2(\kappa+\varepsilon)G_i u_{n,j}, (I+\Delta_{N}^2)u_{n,i} \rangle.
	\end{aligned}
\end{equation}
For $i=1$, the first term could be analyzed as follows:
\begin{equation}\label{3.2.2(7)}
	\begin{aligned}
		&-2\langle G_1(u^n)^3,(I+\Delta_{N}^2)u_{n,1} \rangle  = -2\langle G_1^{\frac{1}{2}}(u^n)^3,G_1^{\frac{1}{2}}(I+\Delta_{N}^2)u_{n,1} \rangle\\
		&  \leq 2\Vert({u}^n)^3\Vert_{2}\cdot\Vert G_1^{\frac{1}{2}}(I+\Delta_{N}^2)u_{n,1}\Vert_{2}\leq16\Vert (u^n)^3\Vert_{2}^2 + \frac{1}{8}(\Vert G_1^{\frac{1}{2}}u_{n,1}\Vert_{2}^2 + \Vert G_1^{\frac{1}{2}}\Delta_{N}^2u_{n,1}\Vert_{2}^2),
	\end{aligned}
\end{equation}
in which the inequality $(a+b)^2\leq2(a^2+b^2)$, summations by parts formula, discrete Cauchy--Schwartz inequality, as well as \Cref{proposition3.1}, have been applied in the analysis. The second term of equality \eqref{3.2.2(6)} could be decomposed into two parts:	
\begin{equation}\label{3.2.2(8)}
	\begin{aligned}
		&2(\kappa-1)\langle G_1u^n,(I+\Delta_{N}^2)u_{n,1} \rangle\leq2(\kappa-1)[\Vert G_1^{\frac{1}{2}}{u}^n\Vert_{2}\cdot(\Vert G_1^{\frac{1}{2}}\Delta_{N}^2u_{n,1}\Vert_{2} +\Vert G_1^{\frac{1}{2}}u_{n,1}\Vert_{2})]\\
		& \leq(\kappa-1)(\Vert G_1^{\frac{1}{2}}{u}^n\Vert_{2}^2 + \Vert G_1^{\frac{1}{2}}u_{n,1}\Vert_{2}^2 + \Vert G_1^{\frac{1}{2}}\Delta_{N} {u}^n\Vert_{2}^2 + \Vert G_1^{\frac{1}{2}}\Delta_{N}u_{n,1}\Vert_{2}^2);
	\end{aligned}
\end{equation}
\begin{equation}\label{3.2.2(9)}
	\begin{aligned}
		&2(1+\varepsilon)\langle G_1u^n,(I+\Delta_{N}^2)u_{n,1} \rangle\leq64\Vert {u}^n\Vert_{2}^2 + \frac{1}{16}\Vert G_1^{\frac{1}{2}}(I+\Delta_{N}^2)u_{n,1}\Vert_{2}^2\\
		&\leq64\Vert {u}^n\Vert_{2}^2 + \frac{1}{8}(\Vert G_1^{\frac{1}{2}}u_{n,1}\Vert_{2}^2 + \Vert G_1^{\frac{1}{2}}\Delta_{N}^2u_{n,1}\Vert_{2}^2).
	\end{aligned}
\end{equation}
Meanwhile, we observe the following inequality, as given by inequalitiy \eqref{eqn:proposition3.2(1)} in \Cref{proposition3.2}: 
\begin{equation}\label{3.2.2(10)}
	\begin{aligned}
		&\Vert G_1^{*}u_{n,1}\Vert_{2}^2 + \Vert G_1^{*}u^n\Vert_{2}^2 +\Vert G_1^{**}u_{n,1} \Vert_{2}^2 + \Vert G_1^{**} {u}^n \Vert_{2}^2\geq \frac{2}{3}(\|G_1^{\frac{1}{2}}u_{n,1}\|_2^2 + \|G_1^{\frac{1}{2}}u^{n}\|_2^2)\\
		& + 
		\frac{1}{4}(\Vert G_1^{\frac{1}{2}}\Delta_{N}u_{n,1}\Vert_{2}^2 +\Vert G_1^{\frac{1}{2}}\Delta_{N}u^n\Vert_{2}^2
		+\Vert G_1^{\frac{1}{2}}\Delta_{N}^2u_{n,1}\Vert_{2}^2 + \Vert G_1^{\frac{1}{2}}\Delta_{N}^2 {u}^n\Vert_{2}^2)\\
		&+ (\kappa-1)(\Vert G_1^{\frac{1}{2}}u_{n,1}\Vert_{2}^2 + \Vert G_1^{\frac{1}{2}}u^n\Vert_{2}^2 + \Vert G_1^{\frac{1}{2}}\Delta_{N}u_{n,1}\Vert_{2}^2 + \Vert G_1^{\frac{1}{2}}\Delta_{N} {u}^n\Vert_{2}^2).
	\end{aligned}
\end{equation}
Subsequently, a substitution of inequalities \eqref{3.2.2(7)}-\eqref{3.2.2(10)} into \eqref{3.2.2(5)} leads to
\begin{equation}\nonumber
	\begin{aligned}
		&\Vert u_{n,1}\Vert_{2}^2 - \Vert {u}^n\Vert_{2}^2 + \Vert\Delta_{N}u_{n,1}\Vert_{2}^2 - \Vert\Delta_{N}{u}^n\Vert_{2}^2 + \frac{\tau}{4}(\Vert G_1^{\frac{1}{2}}\Delta_{N}u_{n,1}\Vert_{2}^2 + \Vert G_1^{\frac{1}{2}}\Delta_{N} {u}^{n}\Vert_{2}^2 \\&+ \Vert G_1^{\frac{1}{2}}\Delta_{N}^2u^n\Vert_2^2)
		+\frac{5}{12}\tau\|G_1^{\frac{1}{2}}u_{n,1}\|_2^2 + \frac{2}{3}\tau\|G_1^{\frac{1}{2}}u^n\|_2^2 \leq64\tau\Vert {u}^n\Vert_{2}^2 + 16\tau\Vert( {u}^n)^3\Vert_{2}^2,
	\end{aligned}
\end{equation}
or in an equivalent manner,
\begin{equation}\nonumber
	\Vert u_{n,1}\Vert_{2}^2 + \Vert \Delta_{N} u_{n,1}\Vert_2^2 \leq (1+64\tau)\Vert {u}^n\Vert_{2}^2 + 16\tau\Vert( {u}^n)^3\Vert_{2}^2 + \|\Delta_{N}u^n\|_2^2\leq(2+64\tau)C_0^2 + 16\tau\tilde{C}_0^4C_0^2,
\end{equation}
where the following inequality has been employed:
\begin{equation}\label{eqn:inequality}
	\Vert (u^n)^3\Vert_2^2 = \|u^n\|_{\infty}^4\cdot\|u^n\|_2^2\leq \tilde{C}_0^4C_0^2.
\end{equation}
Under an $\mathcal{O}(1)$ constraint for the time-step size
\begin{equation}\label{eqn:first_tau}
	\tau\leq \min\left\{\frac{1}{64}, \frac{1}{16}\tilde{C}_0^{-4}\right\},
\end{equation}
we see that
\begin{equation*}
	\Vert u_{n,1}\Vert_{2}^2 + \Vert\Delta_{N}u_{n,1}\Vert_{2}^2\leq 4C_0^2,\quad\text{so that} ~\Vert u_{n,1}\Vert_{2} + \Vert\Delta_{N}u_{n,1}\Vert_2\leq2\sqrt{2}C_0.
\end{equation*}
Therefore, we obtain a rough $\ell^2$, $H_h^2$, and $\ell^{\infty}$ estimates of $u_{n,1}$: 
\begin{equation}\label{eqn:infinity_un1}
	\Vert u_{n,1}\Vert_{2},~\Vert\Delta_{N}u_{n,1}\Vert_{2}\leq 2C_0;\quad\Vert u_{n,1}\Vert_{\infty}\leq \hat{C}(\Vert u_{n,1}\Vert_{2} + \Vert\Delta_{N}u_{n,1}\Vert_{2})\leq2\sqrt{2}\hat{C}C_0=:\tilde{C}_1.
\end{equation}

\subsubsection{Preliminary $\ell^2$ and $H_h^2$ estimates of $G_2^{\frac{1}{2}}u_{n,1}$}
In addition, preliminary estimates of $\Vert G_2^{\frac{1}{2}}u_{n,1}\Vert_{2}$ and $\Vert G_2^{\frac{1}{2}}\Delta_Nu_{n,1}\Vert_2$ are needed to obtain a refined bound of the numerical solution at the next time step. Again, the reformulated numerical system is used in this estimate. Therefore, we take $i=1$ in the system \eqref{eqn:substage_1}-\eqref{eqn:substage_2}.

Taking a discrete $\ell^2$ inner product with equality \eqref{eqn:substage_1} by $G_2(I+\Delta_{N}^2)(u_{n,1}^{*} + u^n)$ results in 
\begin{equation}\label{eqn:3.2.3(1)}
	\begin{aligned}
		&\|G_2^{\frac{1}{2}}u_{n,1}^{*}\|_2^2 - \|G_2^{\frac{1}{2}}u^n\|_2^2 +	\|G_2^{\frac{1}{2}}\Delta_{N}u_{n,1}^{*}\|_2^2 - \|G_2^{\frac{1}{2}}\Delta_{N}u^n\|_2^2 + \tau(\|G^{*}u^n\|_2^2 + \|G^{**}u^n\|_2^2 \\
		&+ \langle G_1L_{\kappa}u^n, G_2(I+\Delta_{N}^2)u_{n,1}^{*}\rangle) = 0.
	\end{aligned}
\end{equation}
Taking a discrete $\ell^2$ inner product with equality \eqref{eqn:substage_2} by $2G_2(I+\Delta_{N}^2)u_{n,1}$ gives
\begin{equation}\label{eqn:3.2.3(2)}
	\begin{aligned}
		&\|G_2^{\frac{1}{2}}u_{n,1}\|_2^2 - \|G_2^{\frac{1}{2}}u_{n,1}^{*}\|_2^2 +
		\|G_2^{\frac{1}{2}}(u_{n,1} - u_{n,1}^{*})\|_2^2 +  \|G_2^{\frac{1}{2}}\Delta_{N}u_{n,1}\|_2^2 - \|G_2^{\frac{1}{2}}\Delta_{N}u_{n,1}^{*}\|_2^2 \\
		&+ \|G_2^{\frac{1}{2}}\Delta_{N}(u_{n,1} - u_{n,1}^{*})\|_2^2
		= 2\tau\langle G_1N_{\kappa}(u^n), G_2(I+\Delta_{N}^2)u_{n,1}\rangle.
	\end{aligned}
\end{equation}
A combination of equalities \eqref{eqn:3.2.3(1)} and \eqref{eqn:3.2.3(2)} yields 
\begin{equation}\label{eqn:3.2.3(3)}
	\begin{aligned}
		& \|G_2^{\frac{1}{2}}u_{n,1}\|_2^2 - \|G_2^{\frac{1}{2}}u^n\|_2^2 + \|G_2^{\frac{1}{2}}(u_{n,1} - u_{n,1}^{*})\|_2^2 + \|G_2^{\frac{1}{2}}\Delta_{N}u_{n,1}\|_2^2 - \|G_2^{\frac{1}{2}}\Delta_{N}u^n\|_2^2 \\
		& + \|G_2^{\frac{1}{2}}\Delta_{N}(u_{n,1} - u_{n,1}^{*})\|_2^2 
		+ \tau(\|G^{*}u^n\|_2^2 +  \|G^{**}u^n\|_2^2
		+\langle G_1L_{\kappa}u^n, G_2(I+\Delta_{N}^2)u_{n,i}^{*}\rangle) \\
		&= 2\tau\langle G_1N_{\kappa}(u^n), G_2(I+\Delta_{N}^2)u_{n,1}\rangle.
	\end{aligned}
\end{equation}
Meanwhile, an application of inequalities \eqref{eqn:proposition3.3(3)}-\eqref{eqn:proposition3.3(4)} in \Cref{proposition3.3} indicates that
\begin{equation}\nonumber
	\begin{aligned}
		&\tau\langle G_1L_{\kappa}u^n, G_2(I+\Delta_{N}^2)u_{n,1}^{*}\rangle + \|G_2^{\frac{1}{2}}(u_{n,1} - u_{n,1}^{*})\|_2^2 + \|G_2^{\frac{1}{2}}\Delta_{N}(u_{n,1} - u_{n,1}^{*})\|_2^2\\
		&\ge \tau(\|G^{*}u_{n,1}\|_2^2 + \|G^{**}u_{n,1}\|_2^2).
	\end{aligned}
\end{equation}
Going back \eqref{eqn:3.2.3(3)}, we arrive at
\begin{equation}\label{eqn:3.2.3(4)}
	\begin{aligned}
		&\|G_2^{\frac{1}{2}}u_{n,1}\|_2^2 - \|G_2^{\frac{1}{2}}u^n\|_2^2 + \|G_2^{\frac{1}{2}}\Delta_{N}u_{n,1}\|_2^2 - \|G_2^{\frac{1}{2}}\Delta_{N}u^n\|_2^2 + \tau(\|G^{*}u^n\|_2^2 + \|G^{*}u_{n,1}\|_2^2 \\
		&+ \|G^{**}u^n\|_2^2 + \|G^{**}u_{n,1}\|_2^2)
		\leq 2\tau\langle G_1N_{\kappa}(u^n), G_2(I+\Delta_{N}^2)u_{n,1}\rangle.
	\end{aligned}
\end{equation}
The RHS of \eqref{eqn:3.2.3(4)} contains two parts:
\begin{equation}\label{eqn:3.2.3(5)}
	\begin{aligned}
		&2\langle G_1N_{\kappa}(u^n), G_2(I+\Delta_{N}^2)u_{n,1}\rangle = -2\langle G_1(u^n)^3,G_2(I+\Delta_{N}^2)u_{n,1}\rangle \\
		&+ 2(\kappa + \varepsilon)\langle G_1 u^n,G_2(I+\Delta_{N}^2)u_{n,1}\rangle.
	\end{aligned}
\end{equation}
The first term could be analyzed as follows:
\begin{equation}\label{eqn:3.2.3(6)}
	\begin{aligned}
		&-2\langle G_1(u^n)^3, G_2(I+\Delta_N^2)u_{n,1}\rangle \leq 2\|(u^n)^3\|_2\cdot\|G_{1,2}^{\frac{1}{2}}(I+\Delta_{N}^2)u_{n,1}\|_2\\
		&\leq 16\tilde{C}_0^4C_0^2 + \frac{1}{16}\|G_{1,2}^{\frac{1}{2}}(I+\Delta_{N}^2)u_{n,1}\|_2^2
		\leq16\tilde{C}_0^4C_0^2 + \frac{1}{8}(\|G_{1,2}^{\frac{1}{2}}u_{n,1}\|_2^2 +  \|G_{1,2}^{\frac{1}{2}}\Delta_{N}^2u_{n,1}\|_2^2).
	\end{aligned}
\end{equation}
The second term of \eqref{eqn:3.2.3(5)} could be decomposed as two parts:
\begin{equation}
	\begin{aligned}
		&2(1+\varepsilon)\langle G_1u^n,G_2(I+\Delta_{N}^2)u_{n,1}\rangle\leq 2(1+\varepsilon)\|u^n\|_2\cdot\|G_{1,2}^{\frac{1}{2}}(I+\Delta_{N}^2)u_{n,1}\|_2\\
		&\leq64C_0 + \frac{1}{16}\|G_{1,2}^{\frac{1}{2}}(I+\Delta_{N}^2)u_{n,1}\|_2^2\leq 64C_0 + \frac{1}{8}(\|G_{1,2}^{\frac{1}{2}}u_{n,1}\|_2^2 + \|G_{1,2}^{\frac{1}{2}}\Delta_{N}^2u_{n,1}\|_2^2);
	\end{aligned}
\end{equation}
\begin{equation}\label{eqn:3.2.3(7)}
	\begin{aligned}
		&2(\kappa-1)\langle G_1u^n, G_2(I+\Delta_{N}^2)u_{n,1}\rangle\leq (\kappa-1)(\|G_{1,2}^{\frac{1}{2}}u^n\|_2^2 + \|G_{1,2}^{\frac{1}{2}}u_{n,1}\|_2^2+\|G_{1,2}^{\frac{1}{2}}\Delta_{N}u^n\|_2^2\\ & + \|G_{1,2}^{\frac{1}{2}}\Delta_{N}u_{n,1}\|_2^2).
	\end{aligned}
\end{equation}
For the two positive $\kappa$-independent terms in \eqref{eqn:3.2.3(7)}, we observe the following inequality, as given by \eqref{eqn:proposition3.2(2)} in \Cref{proposition3.2}:
\begin{equation}\label{eqn:3.2.3(8)}
	\begin{aligned}
		&\|G^{*}u^n\|_2^2 + \|G^{*}u_{n,1}\|_2^2 + \|G^{**}u^n\|_2^2 + \|G^{**}u_{n,1}\|_2^2\ge\frac{2}{3}(\|G_{1,2}^{\frac{1}{2}}u^n\|_2^2 + \|G_{1,2}^{\frac{1}{2}}u_{n,1}\|_2^2\\
		& + \|G_{1,2}^{\frac{1}{2}}\Delta_{N}u^n\|_2^2 + \|G_{1,2}^{\frac{1}{2}}\Delta_{N}u_{n,1}\|_2^2)+ \frac{1}{4}(\|G_{1,2}^{\frac{1}{2}}\Delta_{N}u^n\|_2^2
		+ \|G_{1,2}^{\frac{1}{2}}\Delta_{N}^2u^n\|_2^2 +\|G_{1,2}^{\frac{1}{2}}\Delta_{N}u_{n,1}\|_2^2\\
		&+ \|G_{1,2}^{\frac{1}{2}}\Delta_{N}^2u_{n,1}\|_2^2) + (\kappa-1)(\|G_{1,2}^{\frac{1}{2}}u^n\|_2^2 + \|G_{1,2}^{\frac{1}{2}}u_{n,1}\|_2^2+ \|G_{1,2}^{\frac{1}{2}}\Delta_{N}u^n\|_2^2 + \|G_{1,2}^{\frac{1}{2}}\Delta_{N}u_{n,1}\|_2^2).
	\end{aligned}
\end{equation}
As a result, a substitution of \eqref{eqn:3.2.3(6)}-\eqref{eqn:3.2.3(8)} into \eqref{eqn:3.2.3(4)} yields
\begin{equation}\nonumber
	\begin{aligned}
		&\|G_2^{\frac{1}{2}}u_{n,1}\|_2^2 - \|G_2^{\frac{1}{2}}u^n\|_2^2 + \|G_2^{\frac{1}{2}}\Delta_{N}u_{n,1}\|_2^2 - \|G_2^{\frac{1}{2}}\Delta_{N}u^n\|_2^2 +\frac{2\tau}{3}\|G_{1,2}^{\frac{1}{2}}u^n\|_2^2 +  \frac{5\tau}{12}\|G_{1,2}^{\frac{1}{2}}u_{n,1}\|_2^2\\
		& +\frac{\tau}{4}( \|G_{1,2}^{\frac{1}{2}}\Delta_{N}u^n\|_2^2+ \|G_{1,2}^{\frac{1}{2}}\Delta_{N}^2u^n\|_2^2 + \|G_{1,2}^{\frac{1}{2}}\Delta_{N}u_{n,1}\|_2^2)\leq \tau(64C_0 + 16\tilde{C}_0^4C_0^2).
	\end{aligned}
\end{equation}
Consequently, the following combination of $\ell^2$ and $H_h^2$ bounds become available for $G_2^{\frac{1}{2}}u_{n,1}$:
\begin{equation}\label{eqn:3.2.3_final}
	\|G_2^{\frac{1}{2}}u_{n,1}\|_2^2 +	\|G_2^{\frac{1}{2}}\Delta_{N}u_{n,1}\|_2^2\leq \|G_2^{\frac{1}{2}}u^n\|_2^2 + \|G_2^{\frac{1}{2}}\Delta_{N}u^n\|_2^2 + \tau(64C_0 + 16\tilde{C}_0^4C_0^2).
\end{equation}

\subsubsection{Preliminary $\ell^{\infty}$ estimate of $u^{n+1}$}
We aim to derive a bound for $\Vert u^{n+1} \Vert_{\infty}$. By taking $i=2$, the first term of equality \eqref{3.2.2(6)} can be analyzed as follows:
\begin{equation}\label{3.2.4(1)}
	2\langle G_2u_{n,1}^3, (I+\Delta_{N}^2)u^{n+1}\rangle
	\leq 16\Vert u_{n,1}^3\Vert_{2}^2 + \frac{1}{8}(\Vert G_2^{\frac{1}{2}}u^{n+1}\Vert_{2}^2 + \Vert G_2^{\frac{1}{2}}\Delta_{N}^2u^{n+1}\Vert_{2}^2).
\end{equation}
The second term is also decomposed as two parts: 
\begin{equation}
	\begin{aligned}
		&2(\kappa-1)\langle G_2u_{n,1},(I+\Delta_{N}^2)u^{n+1}\rangle\\
		&\leq(\kappa-1)(\Vert G_2^{\frac{1}{2}}u_{n,1}\Vert_{2}^2 + \Vert G_2^{\frac{1}{2}}{u}^{n+1}\Vert_{2}^2 + \|G_2^{\frac{1}{2}}\Delta_{N}u_{n,1}\|_2^2+ \|G_2^{\frac{1}{2}}\Delta_{N}u^{n+1}\|_2^2)\\
		&\leq (\kappa-1)[(64C_0+16\tilde{C}_0^4C_0^2)\tau + \|G_2^{\frac{1}{2}}u^n\|_2^2 + \|G_2^{\frac{1}{2}}\Delta_{N}u^n\|_2^2 + \|G_2^{\frac{1}{2}}u^{n+1}\|_2 +  \|G_2^{\frac{1}{2}}\Delta_{N}u^{n+1}\|_2^2];\\
		&2(1+\varepsilon)\langle G_2u_{n,1}, (I+\Delta_{N}^2){u}^{n+1}\rangle
		\le64\Vert u_{n,1}\Vert_{2}^2 +  \frac{1}{8}(\Vert G_2^{\frac{1}{2}}{u}^{n+1}\Vert_{2}^2 + \|G_2^{\frac{1}{2}}\Delta_{N}^2u^{n+1}\|_2^2).
	\end{aligned}
\end{equation}
where the preliminary estimate \eqref{eqn:3.2.3_final}  for the intermediate-stage solution has been applied in the last step. Following inequality \eqref{eqn:proposition3.2(1)} in \Cref{proposition3.2}, we have
\begin{equation}\label{3.2.4(2)}
	\begin{aligned}
		&\Vert G_2^{*} {u}^{n+1}\Vert_{2}^2 + \Vert G_2^{*} {u}^{n}\Vert_{2}^2 + \Vert G_2^{**} {u}^{n+1}\Vert_{2}^2 + \Vert G_2^{**} {u}^{n}\Vert_{2}^2 \geq\frac{1}{4}(\Vert G_2^{\frac{1}{2}}\Delta_{N} {u}^{n+1}\Vert_{2}^2 + \Vert G_2^{\frac{1}{2}}\Delta_{N} {u}^{n}\Vert_{2}^2 \\
		&+ \Vert G_2^{\frac{1}{2}}\Delta_{N}^2 {u}^{n+1}\Vert_{2}^2 + \Vert G_2^{\frac{1}{2}}\Delta_{N}^2 {u}^{n}\Vert_{2}^2)
		+ (\kappa-1)(\Vert G_2^{\frac{1}{2}}{u}^{n+1}\Vert_{2}^2 + \Vert G_2^{\frac{1}{2}}{u}^n\Vert_{2}^2 + \Vert G_2^{\frac{1}{2}}\Delta_{N}{u}^{n+1}\Vert_{2}^2\\
		& + \Vert G_2^{\frac{1}{2}}\Delta_{N}{u}^n\Vert_{2}^2) + \frac{2}{3}(\|G_2^{\frac{1}{2}}u^{n+1}\|_2^2 + \|G_2^{\frac{1}{2}}\Delta_{N}u^{n+1}\|_2^2).
	\end{aligned}
\end{equation}
Of course, a substitution of inequalities \eqref{3.2.4(1)}-\eqref{3.2.4(2)} into \eqref{3.2.2(5)} leads to
\begin{equation*}
	\begin{aligned}
		&\|u^{n+1}\|_2^2 - \|u^n\|_2^2 + \|\Delta_{N}u^{n+1}\|_2^2 - \|\Delta_{N}u^n\|_2^2 + \frac{\tau}{4}(\|G_2^{\frac{1}{2}}\Delta_{N}^2u^n\|_2^2 + \Vert G_2^{\frac{1}{2}}\Delta_{N} {u}^{n+1}\Vert_{2}^2\\ 
		&+ \Vert G_2^{\frac{1}{2}}\Delta_{N} {u}^{n}\Vert_{2}^2)
		+ \frac{5}{12}\tau\|G_2^{\frac{1}{2}}u^{n+1}\|_2^2\leq 64\tau\|u_{n,1}\|_2^2 + 16\tau\|u_{n,1}^3\|_2^2 + (\kappa-1)(64C_0+16\tilde{C}_0^4C_0^2)\tau^2 ,
	\end{aligned}
\end{equation*}
or, equivalently,
\begin{equation}\nonumber
	\begin{aligned}
		&\Vert {u}^{n+1}\Vert_{2}^2 + \Vert\Delta_{N} u^{n+1}\Vert_2^2 \leq\Vert {u}^{n}\Vert_{2}^2 + \Vert\Delta_{N}u^n\Vert_{2}^2 +  64\tau\|u_{n,1}\|_2^2 + 16\tau\|u_{n,1}^3\|_2^2 \\
		&+ (\kappa-1)(64C_0+16\tilde{C}_0^4C_0^2)\tau^2
		\leq 2C_0^2 + 256C_0^2\tau + 64\tau\tilde{C}_1^4C_0^2 + (\kappa-1)(64C_0+16\tilde{C}_0^4C_0^2)\tau^2,
	\end{aligned}
\end{equation}
in which the following inequality has been used:
\begin{equation}\nonumber
	\Vert u_{n,1}^3\Vert_2^2  = \|u_{n,1}\|_{\infty}^4\|u_{n,1}\|_2^2\leq 4\tilde{C}_1^4C_0^2.
\end{equation}
Under an $\mathcal{O}(1)$ constraint for the time-step size (more constraint than \eqref{eqn:first_tau})
\begin{equation}\label{eqn:second_tau}
	\tau\leq\min\left\{\frac{1}{256},\frac{1}{64}\tilde{C}_1^{-4},(64\kappa)^{-\frac{1}{2}},\frac{1}{4}\tilde{C}_0^{-2}\kappa^{-\frac{1}{2}}\right\},
\end{equation}
we get 
\begin{equation}\label{eqn:l2h2_un+1}
	\Vert {u}^{n+1}\Vert_{2}^2 + \Vert\Delta_{N} u^{n+1}\Vert_2^2 \leq 6C_0^2,\quad\text{so that}~\Vert {u}^{n+1}\Vert_{2} + \Vert\Delta_{N} u^{n+1}\Vert_2\leq 2\sqrt{3}C_0.
\end{equation}
Also note that $C_0$ is $\kappa$-independent and time-independent. Again, an application of \Cref{lemma3.2} implies the following $\Vert\cdot\Vert_{\infty}$ bound at the next time step
\begin{equation}\label{eqn:infinity_un+1}
	\Vert {u}^{n+1}\Vert_{\infty}\leq \hat{C}\left(\Vert  {u}^{n+1} \Vert_{2} + \Vert \Delta_{N} {u}^{n+1}\Vert_{2}\right)\leq2\sqrt{3}\hat{C}C_0=:\tilde{C}_2 . 
\end{equation}  
Notice that $\tilde{C}_2$ is also a $\kappa$-independent and global-in-time constant. 
%\begin{myremark}
%In the identical analytic framework, another advantage of the ERK\textit{(2,2)} scheme is the less restricted time-step size \eqref{eqn:second_tau}, compared with the ETDRK2 scheme $(\kappa^{-\frac{3}{2}})$.
%\end{myremark}
\subsubsection{Justification of the stabilization parameter $\kappa$ and \textit{a-priori} assumption \eqref{eqn:ass_energy}}
On the other hand, by making comparison between the $\ell^\infty$ bounds for $u_{n,i}~(i=0,1,2)$, given by \eqref{eqn:infinity_un}, \eqref{eqn:infinity_un1}, and \eqref{eqn:infinity_un+1}, respectively, it is clear that $\tilde{C}_2\geq\tilde{C}_1\geq\tilde{C}_0$ for the ERK\textit{(2,2)} scheme \eqref{2.6}. In turn, we could take 
\begin{equation}\label{eqn:kappa}
	\kappa = \max \left\{\frac{|3\tilde{C}_2^2-\varepsilon|}{2},1\right\} , 
\end{equation} 
in which $\tilde{C}_2\ge\beta=\max_{i=0,1,2}\|u_{n,i}\|_{\infty}$. 
Notice that $\kappa$ is an $\mathcal{O}(1)$ constant, and contains no singular dependence on any physical parameter. With this choice of $\kappa$, a fixed constant, we could take the time-step size $\tau$ satisfying \eqref{eqn:second_tau} for ERK\textit{(2,2)}, such that original energy stability becomes available at the next time step by \Cref{theorem3.1}:
\begin{equation}\label{3.2.5_energy}
	E_{N}( {u}^{n+1})\leq E_{N}( {u}^n)\leq E_{N}(u^0) = C_e . 
\end{equation}
This in turn recovers the \textit{a-priori} assumption \eqref{eqn:ass_energy} at the next time step, so that an induction argument can be effectively applied. Therefore, we have proved the main theorem \Cref{theorem3.2}.
\begin{myremark}
	Obviously, the above $\ell^2$, $H_h^2$, and $\ell^{\infty}$ estimates of $u^{n+1}$, namely \eqref{eqn:l2h2_un+1} and \eqref{eqn:infinity_un+1}, turn out to be too rough, since we did not make use of the variational energy structure in the analysis. In fact, to obtain an energy dissipation at the theoretical level, an $\ell^{\infty}$ bound of the numerical solution at the time step $t^{n+1}$ has to be derived, due to the nonlinear term involved. On the other hand, with such a rough bound at hand, we are able to justify the artificial parameter value in \eqref{eqn:kappa}, so that energy stability becomes theoretically available at the next time step. With a theoretical justification of the energy stability analysis, we are able to obtain much sharper $\ell^2$, $H_h^2$, and $\ell^{\infty}$ bounds for the numerical solution $u^{n+1}$.
	
	In more details, with the energy stability result \eqref{3.2.5_energy}, we apply a similar analysis in \eqref{eqn:wise2009} and obtain
	\begin{equation}\nonumber
		\|u^{n+1}\|_2,~\|\Delta_{N}u^{n+1}\|_2\leq C_0:=2\sqrt{C_e + |\Omega|},
	\end{equation}
	which is a global-in-time constant. In turn, a much sharper maximum-norm bound for $u^{n+1}$ also becomes available, with the help of \Cref{lemma3.2}:
	\begin{equation*}
		\|u^{n+1}\|_{\infty}\leq\hat{C}(\|u^{n+1}\|_2 + \|\Delta_{N}u^{n+1}\|_2)\leq 2\hat{C}C_0=:\tilde{C}_3.
	\end{equation*}
	In other words, the $H_h^2$ bound $C_0$ and the $\ell^{\infty}$ bound $\tilde{C}_3$ turns out to be global-in-time constants.
\end{myremark}

\section{Optimal rate convergence analysis}
\label{sec4}
We denote by $u_e$ the exact solution to equation \eqref{1.1}, and assume it satisfies the following regularity:
\begin{equation*}
	u_e\in \mathcal{R} :=  H^3(0, T; C^0) \cap H^2(0, T; H^{m_0}) \cap L^\infty(0, T; H^{m_0+4}).
\end{equation*}
A rigorous error estimate for the ERK\textit{(2,2)} scheme \eqref{2.6} will be derived under this regularity. To this end, the following lemma is needed. 

\begin{mylemma}\label{lemma4.1}
	For any $ {u}, {v}\in\mathcal{M}_{N}$ satisfying $\Vert {u}\Vert_{\infty},~\Vert {v}\Vert_{\infty}\leq\beta$ that is introduced in \Cref{theorem3.1}, we have
	\begin{displaymath}
		\Vert N_{\kappa}( {u}) - N_{\kappa}( {v})\Vert_{2}\leq3\kappa\Vert {u}- {v}\Vert_{2}.
	\end{displaymath}
\end{mylemma}
\begin{proof}
	Since $N_{\kappa}(u) - N_{\kappa}(v) = (u-v)(\varepsilon + \kappa - u^2 - v^2 - uv)$, using $\|u\|_{\infty}\leq\beta$, $\|v\|_{\infty}\leq\beta$, and $\kappa\ge\max\limits_{|\xi|\le\beta}\frac{|3\xi^2-\varepsilon|}{2}$, we obtain the result.
\end{proof}

Meanwhile, we denote $U^n$ as the interpolation values of the projection solution $U_N$ at the grid points at time instant $t_n:~U_{p,q}^n := U_N(x_i,y_j,t_n)$. The initial data is given by
\begin{equation*}
	u_{p,q}^0 = U_{p,q}^0 := U_N(x_p,y_q,t=0).
\end{equation*}
The error grid function is defined as 
\begin{equation}\nonumber
	e^n := U^n - u^n, \quad 0\leq n\leq N_t-1.
\end{equation}
For the proposed ERK\textit{(2,2)} scheme \eqref{2.6}, the convergence result is stated below.
\begin{mytheorem}\label{theorem4.1}
	Given an initial data with sufficient regularity, suppose the unique solution for the SH equation \eqref{1.1} is of regularity class $\mathcal{R}$. Provided that $\tau$ and $h$ are sufficiently small, then, for the ERK\textit{(2,2)} scheme \eqref{2.6}, the following $\ell^\infty(0, T; \ell^2)$ convergence estimate is valid for any $\kappa$ satisfying \eqref{eqn:kappa}:
	\begin{equation}\nonumber
		\Vert u_e(t^n) - u^n\Vert_{2}\leq C(\tau^2+h^{m_0}),\quad\forall n\leq N_t,
	\end{equation}
	where $C>0$ is dependent of $\kappa$ and $\Omega$, but independent of $\tau$ and $h$.
\end{mytheorem}
\begin{proof}
	For the exact solution $u_e$ and its interpolation $U$, a careful consistency analysis implies that
	\begin{equation}\label{4.1}
		\begin{aligned}
			U_{n,1} &= \varphi_0(c_1\tau L_{\kappa}) U^n + c_1\tau\varphi_{1}(c_1\tau L_{\kappa}) N_{\kappa}(U^n), \\
			U^{n+1} &= \varphi_0(\tau L_{\kappa}) U^n + \tau\varphi_{1}(\tau L_{\kappa})N_{\kappa}(U_{n,1})+ \tau \zeta^n,
		\end{aligned}
	\end{equation}
	with $\Vert \zeta^n\Vert_{2}\leq \bar{C}(\tau^2 + h^{m_0})$. Notice that the profile $U_{n,1}$ is constructed, based on the projection solution $U^n$. In turn, subtracting the numerical solution \eqref{2.6} from the consistency estimate \eqref{4.1} yields
	\begin{equation}\label{4.2}
		\begin{aligned}
			e_{n,1} &=  \varphi_0(c_1\tau L_{\kappa}) e^n + c_1\tau\varphi_{1}(c_1\tau L_{\kappa})\tilde{N}_{\kappa}(U^n, {u}^n),\\
			e^{n+1} &= \varphi_0(\tau L_{\kappa}) e^n + \tau\varphi_{1}(\tau L_{\kappa})\tilde{N}_{\kappa}(U_{n,1},u_{n,1}) + \tau \zeta^n, 
		\end{aligned}
	\end{equation}
	with $\tilde{N}_{\kappa}(a,b) = N_{\kappa}(a) - N_{\kappa}(b)$.
	To carry out the error analysis in a more convenient way, we denote $e_{n,i}^{*}=\varphi_0\left(c_i\tau L_{\kappa}\right)e^n$ ($i=1,2$), so that the evolutionary equation \eqref{4.2} could be rewritten as the following two-substage system:
	\begin{align}
		\frac{e_{n,i}^{*} - e^n}{\tau} &= -c_iL_{\kappa}\varphi_{1}(c_i\tau L_{\kappa})e^n,\label{4.3}\\
		\frac{e_{n,i} - e_{n,i}^{*}}{\tau} &= c_i\varphi_{1}(c_i\tau L_{\kappa})\tilde{N}_{\kappa}(U_{n,j}, u_{n,j}) + (i-1)\zeta^n\label{4.4}.
	\end{align}
	Taking a discrete $\ell^2$ inner product with \eqref{4.3} by $e_{n,i}^{*}+e^n$ gives 
	\begin{equation}\label{4.5}
		\Vert e_{n,i}^{*} \Vert_{2}^2 - \Vert e^{n} \Vert_{2}^2 + \tau(\Vert G_i^*e^n\Vert_{2}^2 + \langle G_iL_{\kappa}e^n,e_{n,i}^{*}\rangle)  = 0.
	\end{equation}
	Taking a discrete $\ell^2$ inner product with \eqref{4.4} by $2e_{n,i}$ yields
	\begin{equation}\label{4.6}
		\Vert e_{n,i}\Vert_{2}^2 - \Vert e_{n,i}^{*}\Vert_{2}^2 + \Vert e_{n,i}-e_{n,i}^{*}\Vert_{2}^2 
		= 2\tau\langle G_i\tilde{N}_{\kappa}(U_{n,j}, u_{n,j}),e_{n,i}\rangle  + 2(i-1)\tau\langle  \zeta^n, e_{n,i}\rangle .
	\end{equation}
	As a result, a combination of equalities \eqref{4.5} and \eqref{4.6} leads to 
	\begin{equation}\nonumber
		\begin{aligned}
			&\Vert e_{n,i}\Vert_{2}^2 - \Vert e^{n}\Vert_{2}^2 + \Vert e_{n,i}-e_{n,i}^{*}\Vert_{2}^2 + \tau(\Vert G_i^*e^n\Vert_{2}^2 + \langle G_iL_{\kappa}e^n,e_{n,i}^{*}\rangle)\\
			&= 2\tau\langle G_i\tilde{N}_{\kappa}(U_{n,j}, u_{n,j}),e_{n,i}\rangle  + 2(i-1)\tau\langle  \zeta^n, e_{n,i}\rangle .
		\end{aligned}
	\end{equation} 
	Meanwhile, an application of  \eqref{eqn:proposition3.2(1)} in \Cref{proposition3.2} and \eqref{eqn:proposition3.3(1)} in \Cref{proposition3.3} results in 
	\begin{equation}\nonumber
		\begin{aligned}
			& \|G_i^{*}e_{n,i}\|_2^2\ge\frac{1}{4}\|G_i^{\frac{1}{2}}\Delta_{N}e_{n,i}\|_2^2 + (\kappa-1)\|G_i^{\frac{1}{2}}e_{n,i}\|_2^2 + \frac{2}{3}\|G_i^{\frac{1}{2}}e_{n,i}\|_2^2;\\
			&\tau\langle G_iL_{\kappa}e^n,e_{n,i}^{*}\rangle  + \Vert e_{n,i} - e_{n,i}^{*}\Vert_{2}^2\ge\tau\Vert G_i^*e_{n,i}\Vert_{2}^2.			
		\end{aligned} 
	\end{equation}
	Then we obtain
	\begin{equation}\label{4.7}
		\begin{aligned}
			&\Vert e_{n,i}\Vert_{2}^2 - \Vert e^{n}\Vert_{2}^2 + \tau(\Vert G_i^*e^n\Vert_{2}^2 + \frac{1}{4}\|G_i^{\frac{1}{2}}\Delta_{N}e_{n,i}\|_2^2 + (\kappa-1)\|G_i^{\frac{1}{2}}e_{n,i}\|_2^2 + \frac{2}{3}\|G_i^{\frac{1}{2}}e_{n,i}\|_2^2)\\
			&\leq2\tau\langle G_i\tilde{N}_{\kappa}(U_{n,j}, u_{n,j}), e_{n,i}\rangle  + 2(i-1)\tau\langle  \zeta^n, e_{n,i}\rangle .
		\end{aligned}
	\end{equation}
	Moreover, an application of \Cref{lemma4.1} implies that 
	\begin{equation}\nonumber
		\Vert\tilde{N}_{\kappa}(U_{n,j}, u_{n,j}) \Vert_{2}\leq3\kappa\Vert e_{n,j}\Vert_{2},
	\end{equation}
	which would be used for the derivation of the following nonlinear inner product term estimate ($i=1$):
	\begin{equation}\label{eqn:4(1)}
		2\langle G_1\tilde{N}_{\kappa}(U^n, {u}^n),e_{n,1}\rangle\leq 36\kappa^2\Vert e^n\Vert_{2}^2 + \frac{1}{4}\Vert G_1^{\frac{1}{2}}e_{n,1} \Vert_{2}^2.
	\end{equation}	
	Its substitution into inequality \eqref{4.7} yields
	\begin{equation}\nonumber
		\Vert e_{n,1}\Vert_{2}^2 - \Vert e^{n}\Vert_{2}^2
		\leq36\kappa^2\tau\Vert e^n\Vert_{2}^2.
	\end{equation}
	Provided that $\tau\leq(6\kappa)^{-2}$, a preliminary $\ell^2$ error estimate for the intermediate-stage error solution $u_{n,1}$ is obtained
	\begin{equation}\nonumber
		\Vert e_{n,1}\Vert_{2} \leq\left[1 + 36\kappa^2\tau\right]^{\frac{1}{2}}\Vert e^n\Vert_{2}\leq\sqrt{2}\Vert e^n\Vert_{2}.
	\end{equation}
	
	For $i=2$, the analysis for the nonlinear error inner product term of inequality \eqref{eqn:4(1)} could be similarly established: 
	\begin{equation}\label{4.9}
		2\langle G_2\tilde{N}_{\kappa}(U_{n,1},u_{n,1}),e^{n+1}\rangle 
		\leq72\kappa^2\Vert e^{n}\Vert_{2}^2 + \frac{1}{4}\Vert G_2^{\frac{1}{2}}e^{n+1}\Vert_{2}^2,
	\end{equation}
	in which the estimate $\Vert e_{n,1}\Vert_2\leq\sqrt{2}\Vert e^n\Vert_2$ has been applied in the last step.
	\noindent A bound for the truncation error inner product term is more straightforward:
	\begin{equation}\label{4.10}
		2\langle \zeta^n,e^{n+1} \rangle \leq\Vert \zeta^n \Vert_{2}^2+\Vert e^{n+1}\Vert_{2}^2.
	\end{equation}
	Afterwards, a substitution of inequalities \eqref{4.9} and \eqref{4.10} into \eqref{4.7} leads to 
	\begin{equation}\nonumber
		\Vert e^{n+1}\Vert_{2}^2 - \Vert e^{n}\Vert_{2}^2 \leq 72\kappa^2\tau\Vert e^{n}\Vert_{2}^2+ \tau\Vert e^{n+1}\Vert_{2}^2 + \tau\Vert \zeta^n\Vert_{2}^2.
	\end{equation}
	In turn, an application of the Gr\"{o}nwall's inequality results in the desired convergence estimate:
	\begin{equation}\nonumber
		\Vert e^{n+1}\Vert_{2}\leq C(\tau^2+h^{m_0}),
	\end{equation}
	due to the fact that $\Vert \zeta^n\Vert_{2}\leq \bar{C}(\tau^2+h^{m_0})$. This validates the convergence estimate.
\end{proof}

\section{Numerical experiments: Comparison with other energy-stable methods}
\label{sec5}
In this section, we present a few 2-D numerical results for the SH equation \eqref{1.1}), to demonstrate the efficiency, accuracy, and long-time performance of the ERK\textit{(2,2)} scheme \eqref{2.6}. To preserve energy stability, the condition $\kappa=\max\left\{\frac{|3\tilde{C}_2^2-\varepsilon|}{2},1\right\}$
%, with $\tilde{C}_2\ge\beta=\max_{i=0,1,2}\|u_{n,i}\|_{\infty}$, 
is theoretically required in \Cref{theorem3.2}. Although there is no maximum principle for the SH equation \eqref{1.1}, in practice we observe that the numerical solution is always bounded in $[-1,1]$ during the whole simulation process. Therefore, it suffices to set $\kappa=2$ in the computation. 

\subsection{Convergence test}
The computational domain is taken as $\Omega =[0,32]^2$, and the following smooth initial data is enforced: 
\begin{equation}\nonumber
	u_0(x,y) = 0.01\times[\cos(\pi x) + \cos(\pi y) + \cos(0.25\pi x) + \cos(0.25\pi y)] , 
\end{equation}
on the uniform mesh $N=N_x=N_y$ with $N=256$. The final time is set as $T=5$. To obtain the numerical errors, we set the numerical solution obtained by the ERK\textit{(2,2)} scheme with $\tau = 0.1\times2^{-9}$ as the reference one. Afterwards, we perform the numerical simulations of the first-order ETD1, IMEX1, and second-order ERK\textit{(2,2)}, ETDRK2, IMEX-RK(2,2) schemes using time-step sizes $\tau = 2^{-k} (k=1,\ldots,9$), with two different parameters, $\varepsilon=0.25$ and $0.025$. The results are displayed in \Cref{fig_err}, and through observation we observe that the exponential-type schemes consistently outperform the non-exponential-type ones, in terms of computational accuracy and efficiency. Further, ERK\textit{(2,2)} does the best.%The last but not least, it could be observed that, the numerical error of the two schemes is positively correlated with the parameter $\varepsilon$.
\begin{figure}[htbp]
	\centering
	%	\scalebox{0.8}{
		\vspace{-2mm}
		\mbox{\includegraphics[width=0.5\textwidth]{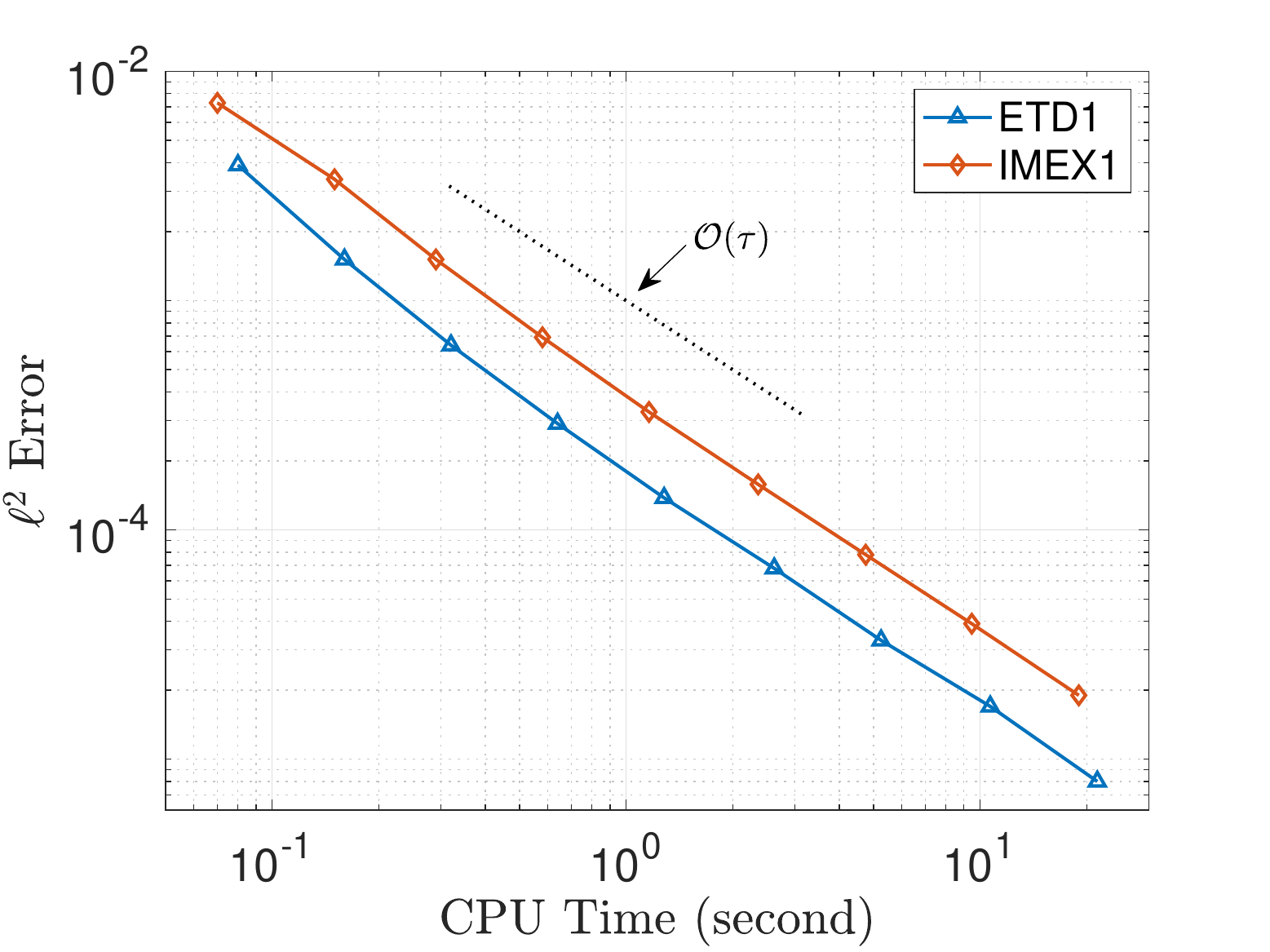}}\hspace{-4mm}
		\mbox{\includegraphics[width=0.5\textwidth]{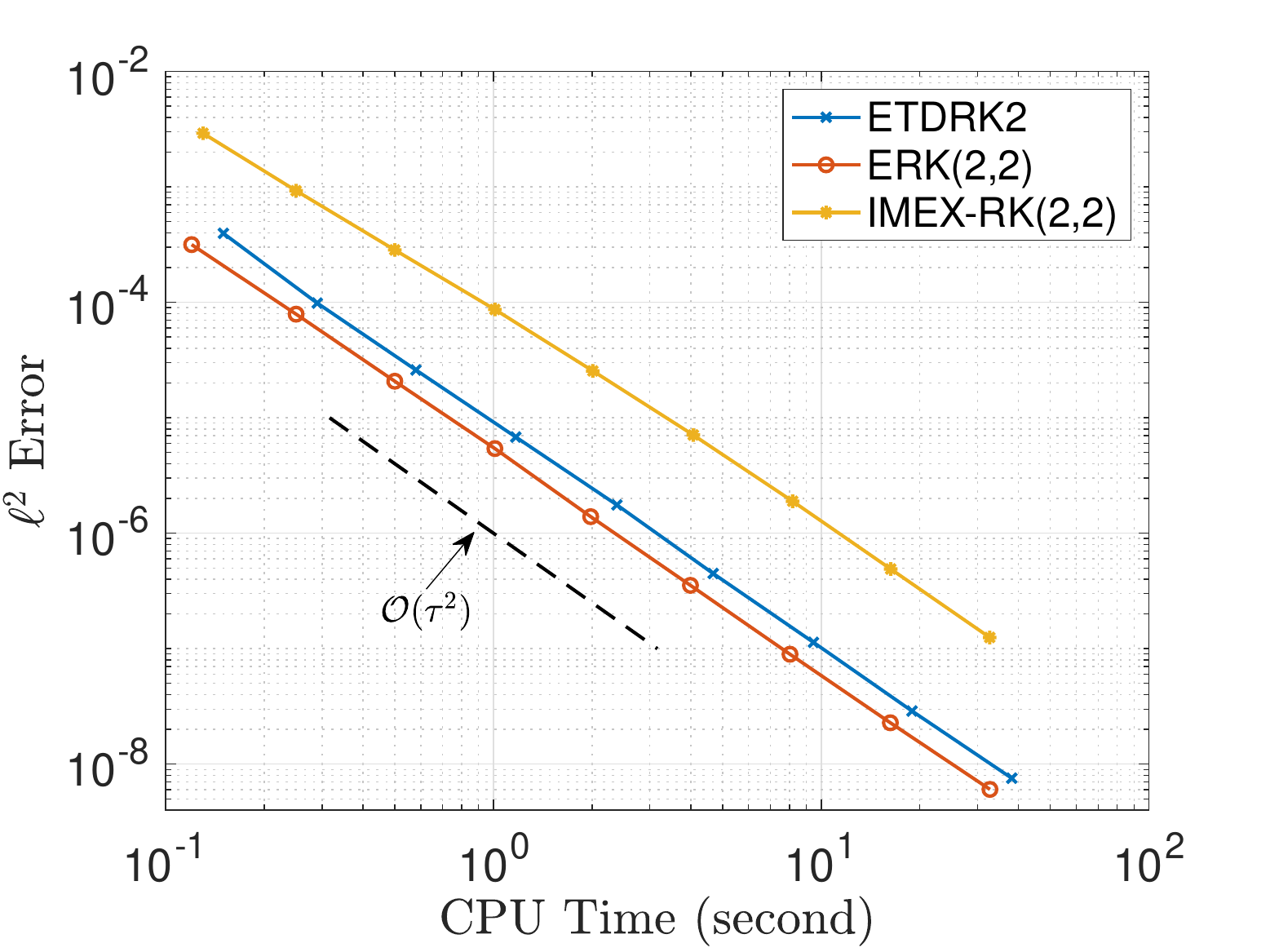}}
		\mbox{\includegraphics[width=0.5\textwidth]{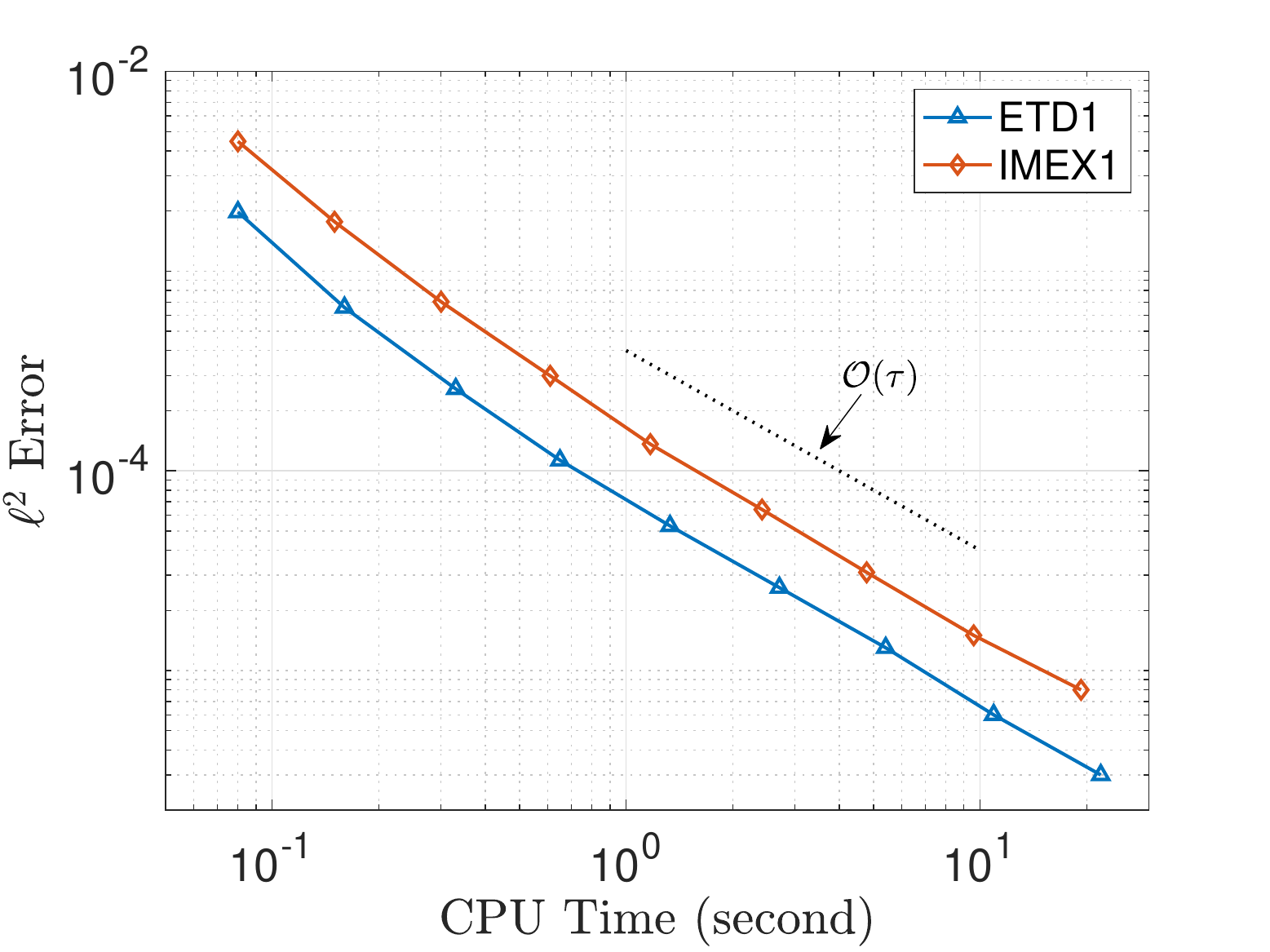}}\hspace{-4mm}
		\mbox{\includegraphics[width=0.5\textwidth]{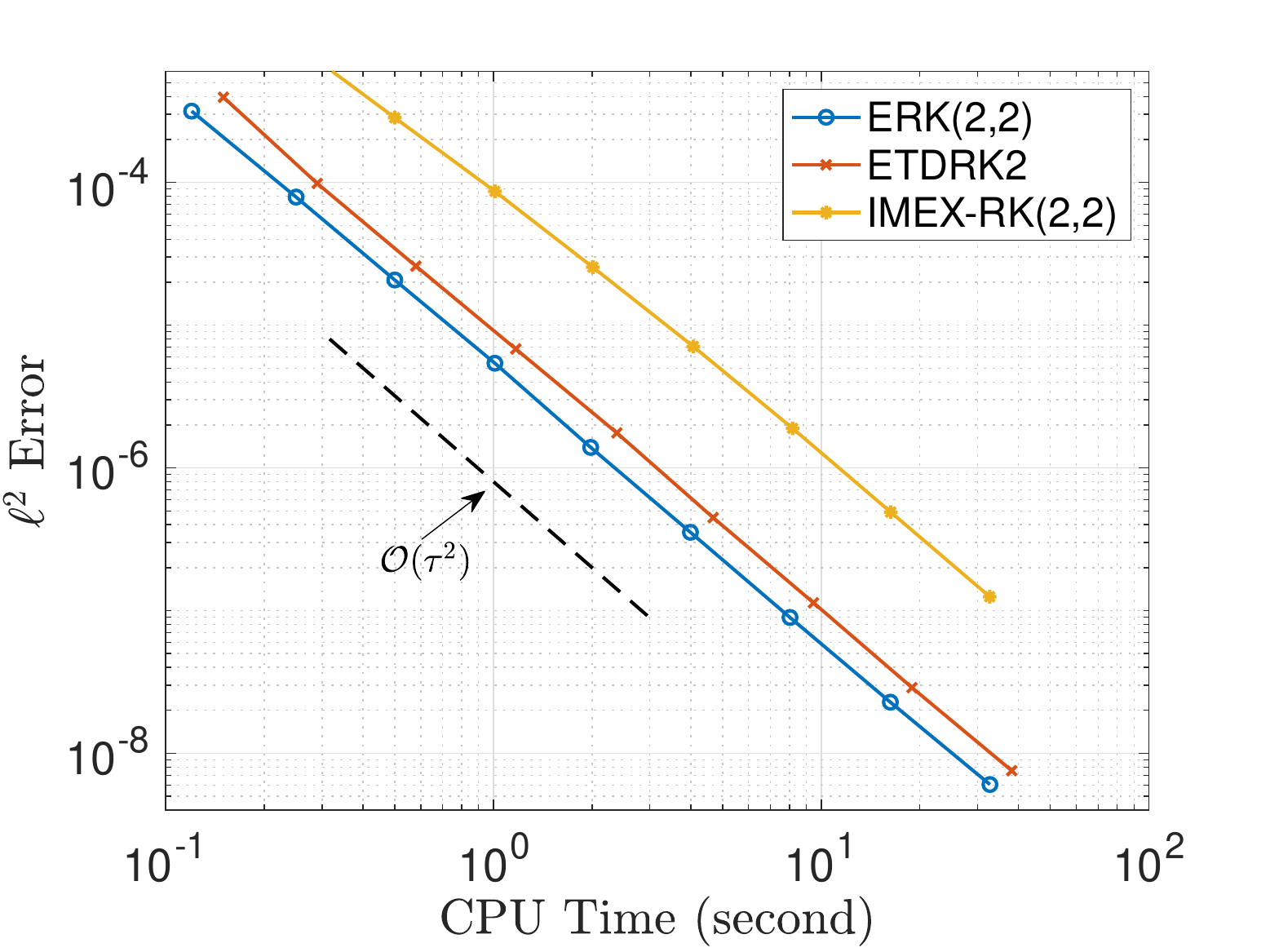}}
		\vspace{-2mm}
		%	}
	\caption{Convergence of the first- and second-order Fourier pseudo-spectral schemes in time with fixed $\tau$ (left) and $\tau^2$ (right) for the 2-D SH equation. Top: $\varepsilon=0.25$. Bottom: $\varepsilon=0.025$. It is seen that the numerical error magnitude and computational cost of the ERK\textit{(2,2)} scheme are smaller than those of both the ETDRK2 and IMEX-RK(2,2) schemes, although they share the same convergence order.}
	\label{fig_err}
\end{figure}
\subsection{Energy stable test}
We simulate the pattern formation and evolution of the SH equation \eqref{1.1}, which arises in the Rayleigh--B{\'e}nard convection. We conduct the simulation on a rectangular domain $\Omega = [0, 100]^2$ from $T=0$ to $100$, subject to the following initial condition:
\begin{equation}\nonumber
	u_0(x,y) = 0.1 + 0.02\times\cos\left(\frac{\pi x}{100}\right) \sin\left(\frac{\pi y}{100}\right) + 0.05\times\sin\left(\frac{\pi x}{20}\right) \cos\left(\frac{\pi y}{20}\right).
\end{equation}
%where $\mathrm{rand}(x,y)$ is the random number uniformly distributed between $-1$ and $1$.

Setting the spatial mesh $256\times256$, we compare the original energy evolution by different time step sizes and energy-stable methods, which are depicted in \Cref{fig_energy}. In actual implementation, we discover that ERK(2,2) allows a rather mild time-step restriction, which will be useful in future work on large-scale scientific computing.

\begin{figure}[htbp]
	\centering
	%	\scalebox{0.8}{
		\vspace{-2mm}
		\subfloat{\includegraphics[width=0.48\textwidth]{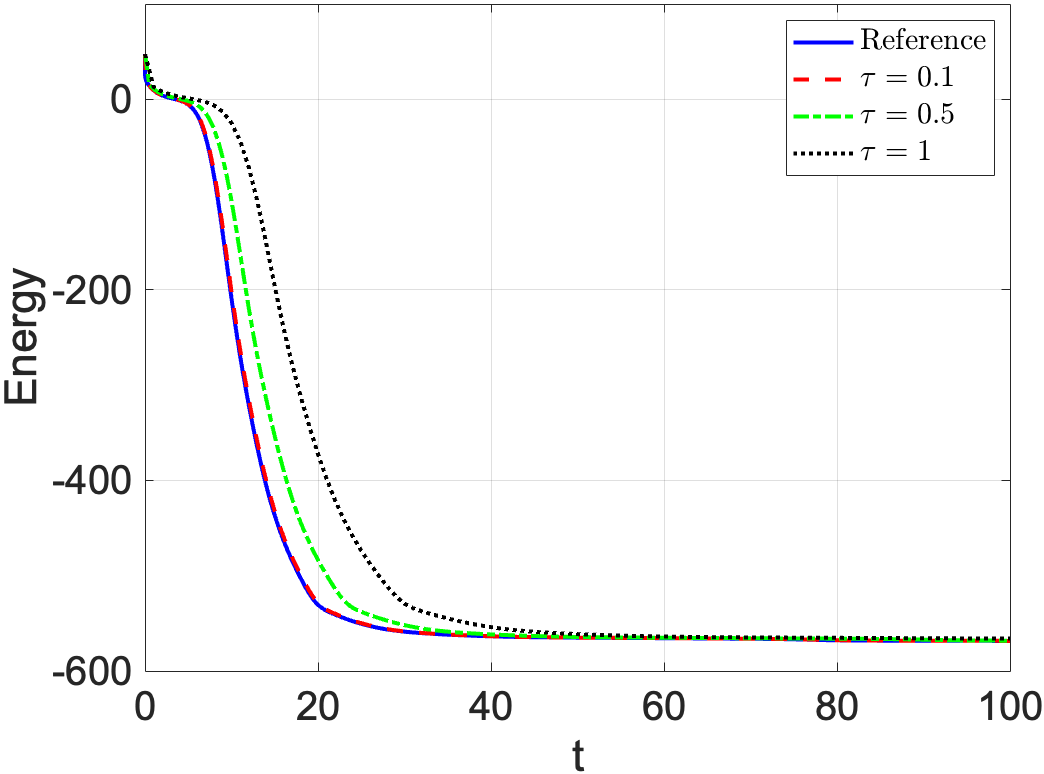}}
		\subfloat{\includegraphics[width=0.48\textwidth]{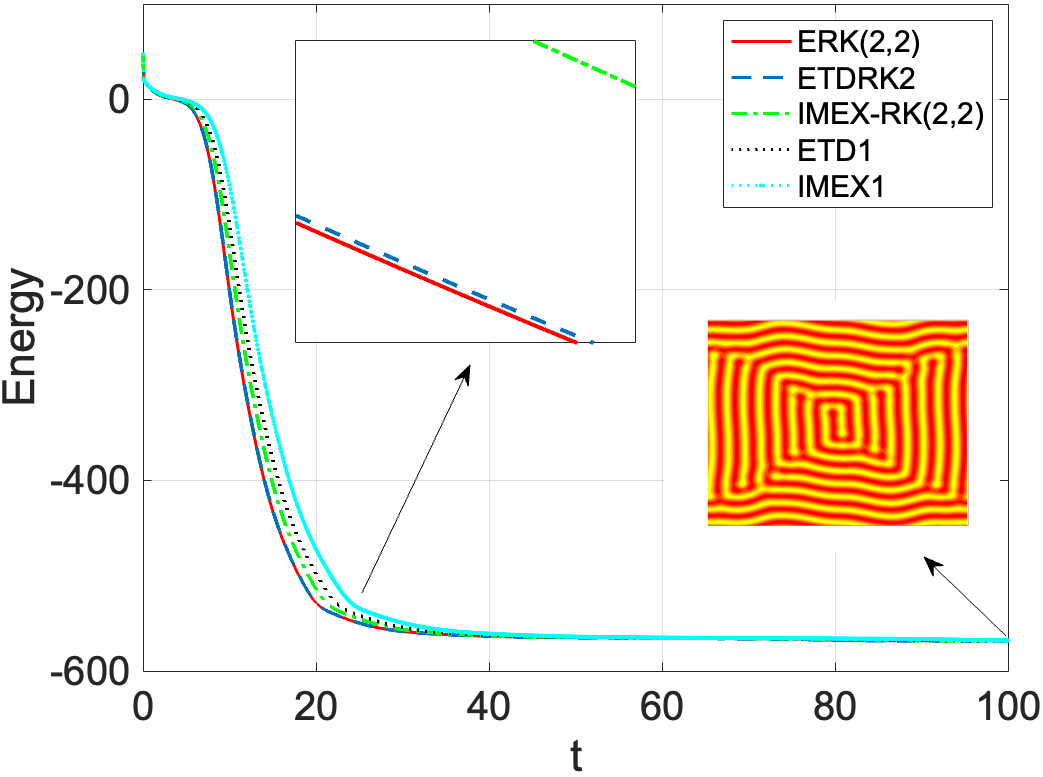}}
		\vspace{-2mm}
		%	}
	\caption{The evolution of the original energy using ERK(2,2) with different time-step sizes (left) and a fixed size $\tau=0.1$ with various energy-stable methods (right) is shown. It is observed that, while only the red dashed line ($\tau=0.1$) approximately matches the reference line, the other two lines also exhibit the same energy-decreasing trend and eventually converge to the same steady state. From the right subplot, it can be seen that all methods maintain discrete energy stability over extended periods, but ERK(2,2) reaches the steady state more rapidly}
	\label{fig_energy}
\end{figure}
\subsection{Polycrystal growth in a supercooled liquid}
In the existing studies~\cite{swift1977, yang2021linear, sujian}, the polycrystal growth in a supercooled liquid was considered as an important 2-D benchmark test. Here, we look at the growth of three crystal nucleuses with the following initial data:
\begin{equation}\nonumber
	u_0(x,y) = 0.287 + \alpha\times\mathrm{rand}(x,y),
\end{equation}
in which $\mathrm{rand}(x,y)$ is the random number uniformly distributed between $-1$ and $1$, and $\alpha$ takes the values of $0.1,~0.2$, and $0.4$ for three crystal nucleuses locating at $(375,125)$, $(375,375)$, and $(125,250)$, respectively. The length of each nucleus is $10$. We also set the computational domain, spatial resolution, and time-step size as $(0,500)^2$, $512\times512$, and $0.5$, respectively. In this test, the parameter $\varepsilon$ is chosen to be $0.25$. \Cref{fig5} displays snapshots of the crystal micro-structure at several time instants. 
\begin{figure}[htbp]
	\centering
	\vspace{-2mm}
	\subfloat{\includegraphics[width=0.32\textwidth]{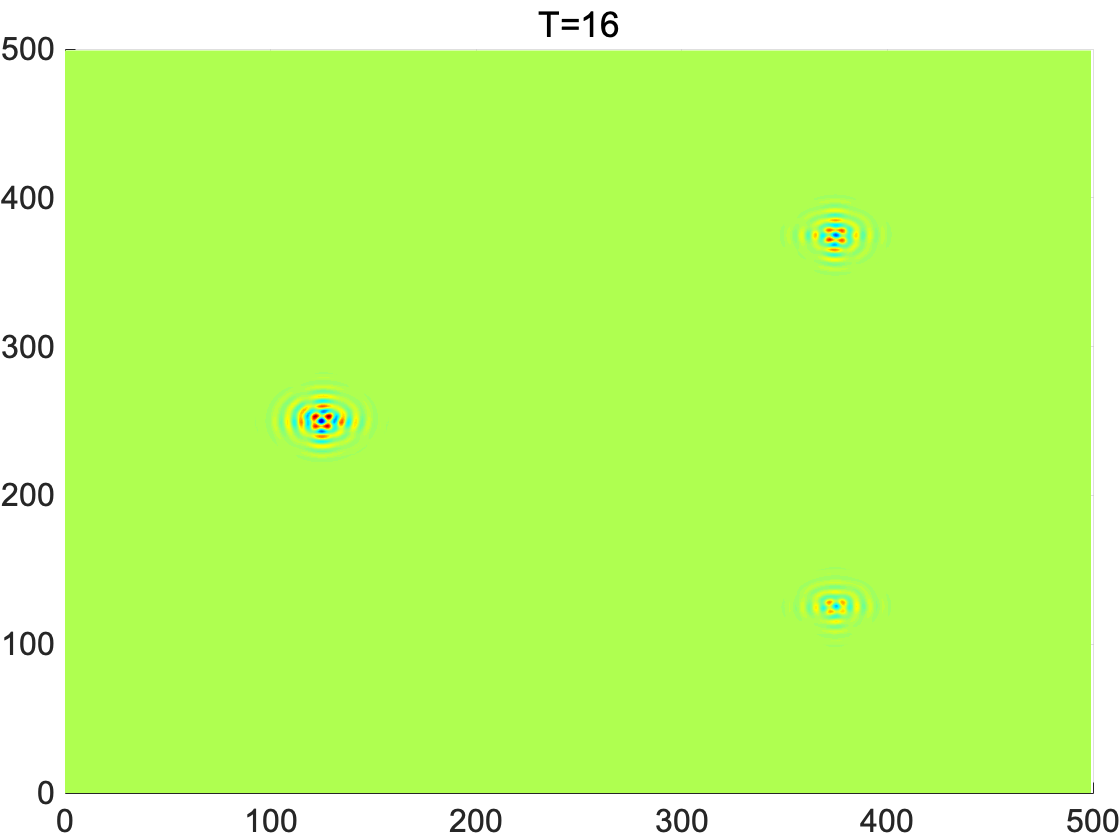}}\hspace{-1mm}
	\subfloat{\includegraphics[width=0.32\textwidth]{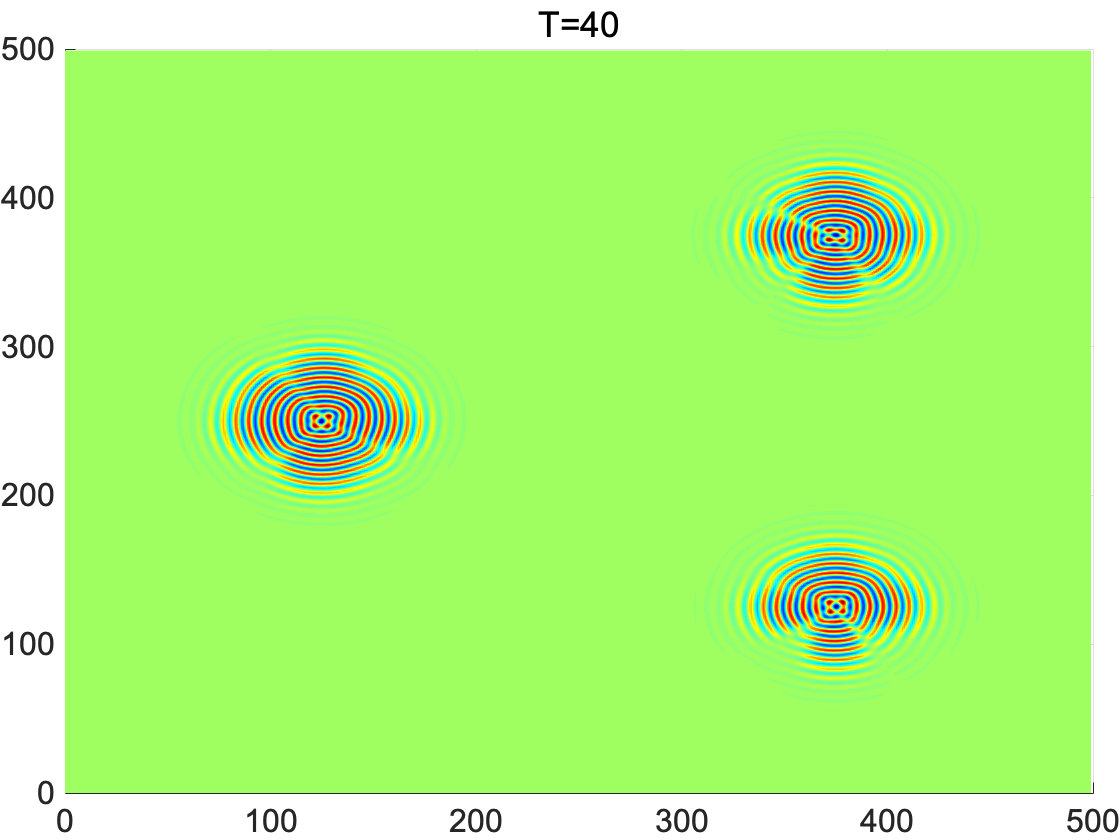}}\hspace{-1mm}
	\subfloat{\includegraphics[width=0.32\textwidth]{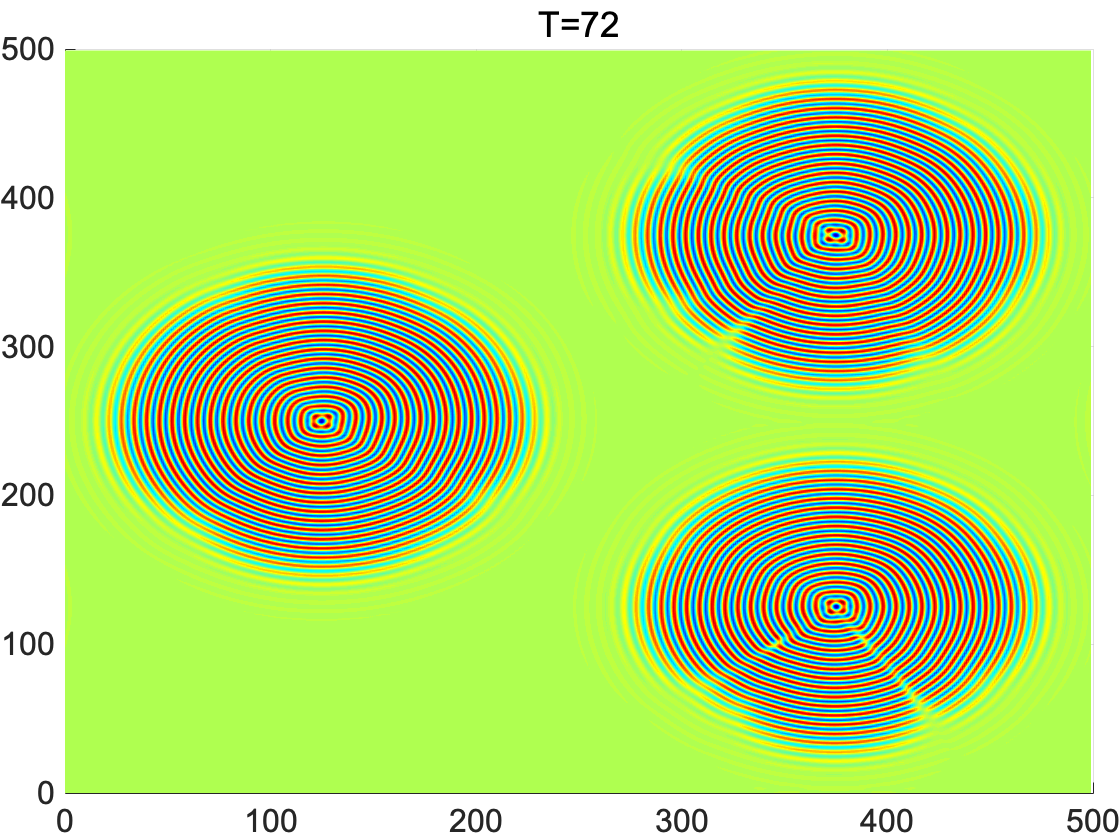}}
	\vspace{-3mm}
	\subfloat{\includegraphics[width=0.32\textwidth]{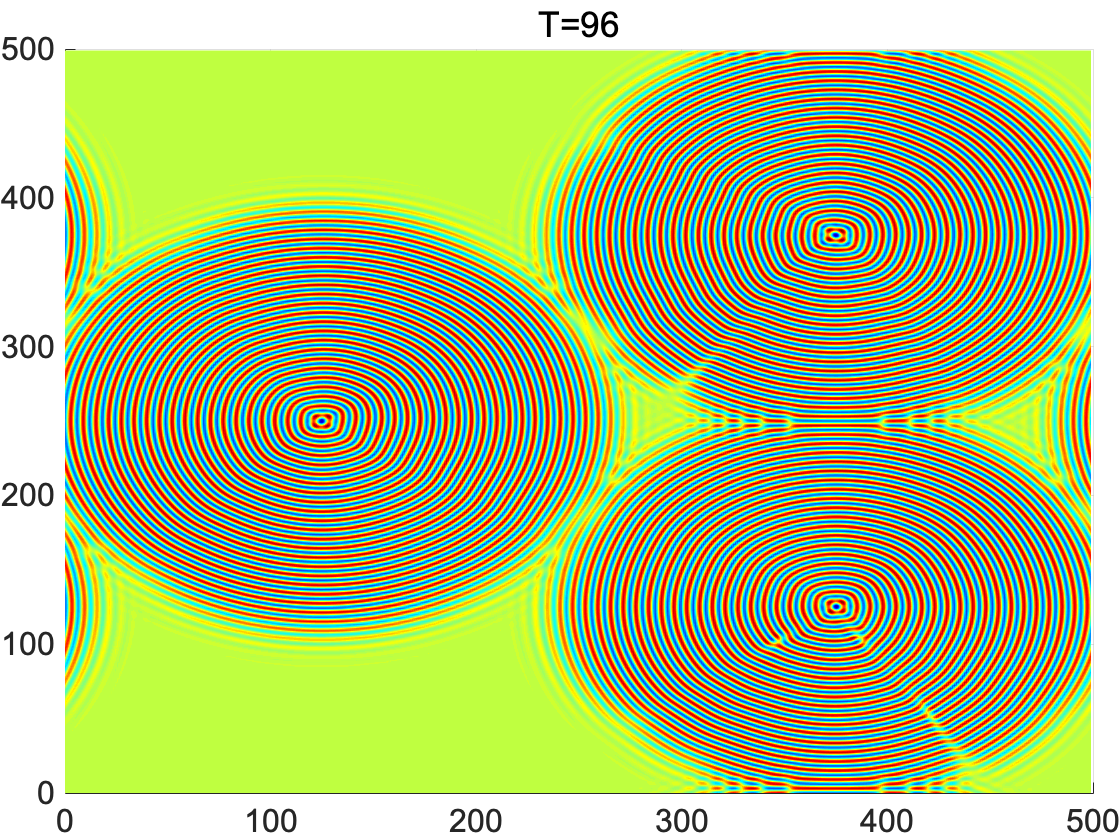}}\hspace{-1mm}
	\subfloat{\includegraphics[width=0.32\textwidth]{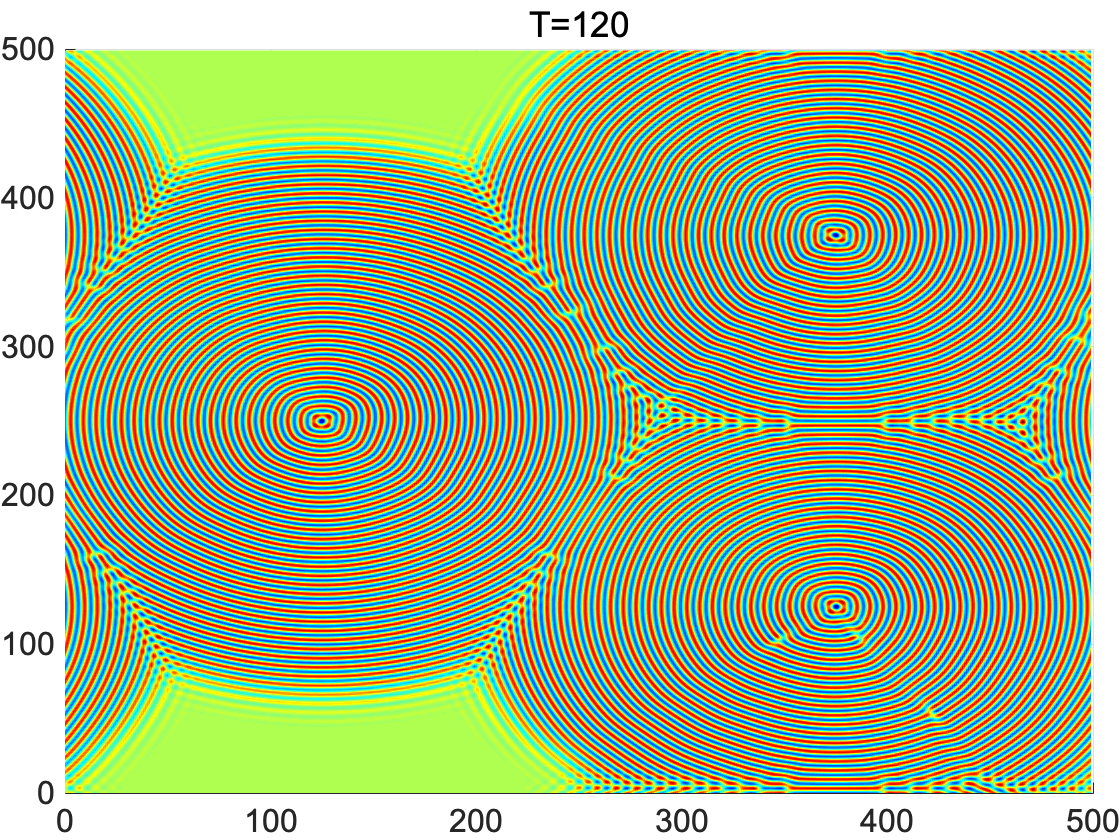}}\hspace{-1mm}
	\subfloat{\includegraphics[width=0.32\textwidth]{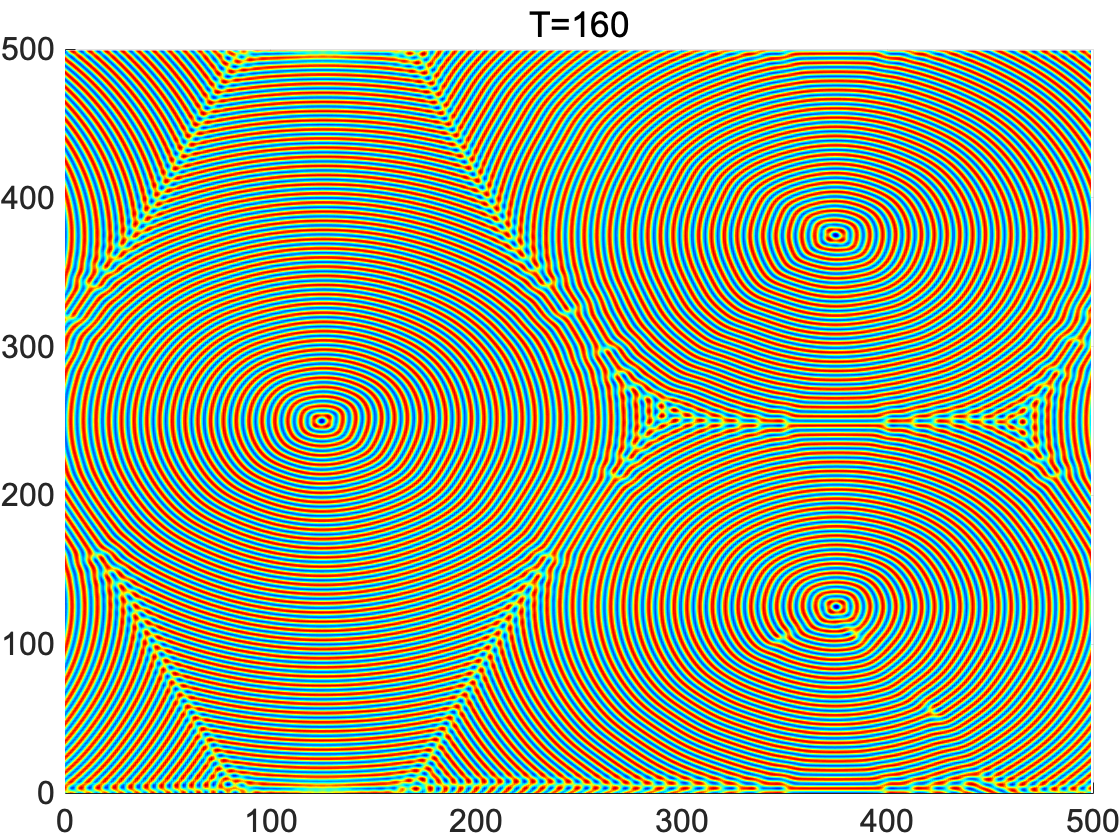}}
	\vspace{-2mm}
	\caption{Evolution of 2-D polycrystal growth in a supercooled liquid at $T = 16,~40,~72,~96,~120$, and $160$ computed by ERK(2,2). It can be seen that three different nuclei grains grow and eventually become sufficiently large to form grain boundaries.}
	\label{fig5}
\end{figure}
\section{Concluding Remarks}
\label{sec6}
Taking the analysis of the SH equation as a demonstration, we have proposed a new strategy for proving discrete energy stability without any preconditions. Further, to numerically overcome the difficulty caused by strong stiffness, we have devised a stabilization exponential Runge--Kutta (ERK) scheme, which is equipped with an appropriate linear stabilization term, preserving the dissipation property of the original energy. We determined the $\ell^{\infty}$ bounds of the numerical solution at all ERK stages so that global-in-time energy stability becomes available. Such an $\ell^\infty$ analysis was accomplished by the $\ell^2$ and $H_h^2$ estimates of the numerical solution at each ERK stage, with the help of summation-by-parts formulas in the Fourier pseudo-spectral space, the discrete Sobolev inequality and elliptic regularity, as well as the eigenvalue estimates in the Fourier space. The global-in-time stability analysis for the original energy is the \textbf{first such result} for SH-type equations. Furthermore, we have provided an optimal rate convergence analysis for the proposed scheme, under a sufficient regularity assumption for the exact solution. A few numerical results have also been presented in this paper. In the convergence test, it was discovered that ERK\textit{(2,2)} outperforms ETDRK2 and IMEX-RK(2,2), in terms of efficiency and accuracy. Besides, the long-time numerical simulation of pattern formation and further evolution have also validated the robustness of the proposed scheme. 
\section*{Acknowledgements}

This research was supported by the National Key R\&D Program of China [Grant No. 2020YFA0709803], the National Natural Science Foundation of China [Grant Nos. 12071481 and 12271523], the Defense Science Foundation of China [Grant No. 2021-JCJQ-JJ-0538], and the Science and Technology Innovation Program of Hunan Province [Grant Nos. 2021RC3082 and 2022RC1192].

%\end{appendices}

%\begin{figure}[h]
%\includegraphics[scale=0.7]{conedrawings.pdf}
%\caption{Map from cone to square.}
%\label{conefig}
%\end{figure}

%  \bibliographystyle{abbrv}

%\bibliography{References}

\begin{thebibliography}{99}

\vspace{-.17cm}


\bibitem{chen2012linear}
{\sc W.~Chen, S.~Conde, C.~Wang, X.~Wang, and S.~M. Wise}, {\em A linear energy stable scheme for a thin film model without slope selection}, Journal of Scientific Computing, 52 (2012), pp.~546--562.

\bibitem{chen2014linear}
{\sc W.~Chen, C.~Wang, X.~Wang, and S.~M. Wise}, {\em A linear iteration algorithm for a second-order energy stable scheme for a thin film model without slope selection}, Journal of Scientific Computing, 59 (2014), pp.~574--601.

\bibitem{cheng2008efficient}
{\sc M.~Cheng and J.~A. Warren}, {\em An efficient algorithm for solving the phase field crystal model}, Journal of Computational Physics, 227 (2008), pp.~6241--6248.

\bibitem{cross1993}
{\sc M.~C. Cross and P.~C. Hohenberg}, {\em Pattern formation outside of equilibrium}, Reviews of modern physics, 65 (1993), p.~851.

\bibitem{dehghan2016}
{\sc M.~Dehghan and V.~Mohammadi}, {\em The numerical simulation of the phase field crystal {(PFC)} and modified phase field crystal {(MPFC)} models via global and local meshless methods}, Computer Methods in Applied Mechanics and Engineering, 298 (2016), pp.~453--484.

\bibitem{gottlieb2012long}
{\sc S.~Gottlieb, F.~Tone, C.~Wang, X.~Wang, and D.~Wirosoetisno}, {\em Long time stability of a classical efficient scheme for two-dimensional {Navier--Stokes} equations}, SIAM Journal on Numerical Analysis, 50 (2012), pp.~126--150.

\bibitem{gottlieb2012stability}
{\sc S.~Gottlieb and C.~Wang}, {\em Stability and convergence analysis of fully discrete {Fourier} collocation spectral method for {3-D viscous Burgers’} equation}, Journal of Scientific Computing, 53 (2012), pp.~102--128.

\bibitem{hochbruck2005}
{\sc M.~Hochbruck and A.~Ostermann}, {\em Explicit exponential {Runge--Kutta} methods for semilinear parabolic problems}, SIAM Journal on Numerical Analysis, 43 (2005), pp.~1069--1090.

\bibitem{hohenberg1992}
{\sc P.~Hohenberg and J.~Swift}, {\em Effects of additive noise at the onset of {Rayleigh-B{\'e}nard} convection}, Physical Review A, 46 (1992), p.~4773.

\bibitem{hutt2005}
{\sc A.~Hutt and F.~M. Atay}, {\em Analysis of nonlocal neural fields for both general and gamma-distributed connectivities}, Physica D: Nonlinear Phenomena, 203 (2005), pp.~30--54.

\bibitem{hutt2008}
{\sc A.~Hutt, A.~Longtin, and L.~Schimansky-Geier}, {\em Additive noise-induced {Turing} transitions in spatial systems with application to neural fields and the {SH} equation}, Physica D: Nonlinear Phenomena, 237 (2008), pp.~755--773.

\bibitem{lee2019energy}
{\sc H.~G. Lee}, {\em An energy stable method for the {SH} equation with quadratic--cubic non-linearity}, Computer Methods in Applied Mechanics and Engineering, 343 (2019), pp.~40--51.

\bibitem{lee2016}
{\sc H.~G. Lee and J.~Kim}, {\em A simple and efficient finite difference method for the phase-field crystal equation on curved surfaces}, Computer Methods in Applied Mechanics and Engineering, 307 (2016), pp.~32--43.

\bibitem{LiX2024a}
{\sc X.~Li and Z.~Qiao}, {\em A second-order, linear, $L^\infty$-convergent, and energy stable scheme for the phase field crystal equation}, SIAM Journal on Scientific Computing, 46 (2024), pp.~A429--A451. 

\bibitem{liu2023efficient}
{\sc Z.~Liu and C.~Chen}, {\em On efficient semi-implicit auxiliary variable methods for the six-order {SH} model}, Journal of Computational and Applied Mathematics, 419 (2023), p.~114730.

\bibitem{rosa2000}
{\sc R.~Rosa, J.~Pont{\`e}s, C.~Christov, F.~M. Ramos, C.~R. Neto, E.~L. Rempel, and D.~Walgraef}, {\em Gradient pattern analysis of {SH} dynamics: phase disorder characterization}, Physica A: Statistical Mechanics and its Applications, 283 (2000), pp.~156--159.

\bibitem{sun2023family}
{\sc J.~Sun, H.~Zhang, X.~Qian, and S.~Song}, {\em {A family of Structure-preserving Exponential Time Differencing Runge--Kutta Schemes for the Viscous Cahn--Hilliard Equation}}, Journal of Computational Physics,  (2023), p.~112414.

\bibitem{fuzhaohui}
{\sc Z.~Fu and J.~Yang}, {\em {Energy-decreasing exponential time differencing Runge–Kutta methods for phase-field models}}, Journal of Computational Physics,  (2022), p.~110943.

\bibitem{swift1977}
{\sc J.~Swift and P.~C. Hohenberg}, {\em Hydrodynamic fluctuations at the convective instability}, Physical Review A, 15 (1977), p.~319.

\bibitem{wise2009energy}
{\sc S.~M. Wise, C.~Wang, and J.~S. Lowengrub}, {\em An energy-stable and convergent finite-difference scheme for the phase field crystal equation}, SIAM Journal on Numerical Analysis, 47 (2009), pp.~2269--2288.

\bibitem{yang2021linear}
{\sc J.~Yang and J.~Kim}, {\em Linear and energy stable schemes for the {SH} equation with quadratic-cubic non-linearity based on a modified scalar auxiliary variable approach}, Journal of Engineering Mathematics, 128 (2021), p.~21.

\bibitem{yang2021conservative}
{\sc J.~Yang, Z,~Tan, and J.~Kim}, {\em High-order time-accurate, efficient, and structure-preserving numerical methods for the conservative {SH} model}, Computers \& Mathematics with Applications, 102 (2021), pp.~160--174.

\bibitem{sujian}
{\sc J.~Su, W.~Fang, Q.~Yu, and Y.~Li}, {\em Numerical simulation of {SH} equation by the fourth-order compact scheme}, Computational and Applied Mathematics, 38 (2019), p.~54.

\bibitem{fengx}
{\sc X.~Feng, T.~Tang, and J.~Yang}, {\em Stabilized {Crank-Nicolson/Adams-Bashforth} schemes for phase field models}, East Asain Journal on Applied Mathematics, 3 (2013), pp.~59--80.

\bibitem{trancated1}
{\sc N.~Condette, C.~Melcher, and E.~Suli}, {\em Spectral approximation of pattern-forming nonlinear evolution equations with double-well potentials of quadratic growth}, Mathematics of Computation, 80 (2011), pp.~205--223.
\bibitem{trancated2}
{\sc S.~Pei, Y.~Hou, and B.~You}, {\em A linearly second-order energy stable scheme for the phase field crystal model}, Applied Numerical Mathematics, 140 (2019), pp.~134--164.

\bibitem{trancated3}
{\sc J.~Shen and J.~Yang}, {\em Numerical approximations of {Allen--Cahn and Cahn--Hilliard} equations}, Discrete and Continuous Dynamical Systems, 28 (2010), pp.~1669--1691.

\bibitem{trancated4}
{\sc M.~Cui, Y.~Niu, and Z.~Xu}, {\em A second order exponential time differencing multi-step energy stable scheme for {SH} equation with quadratic-cubic nonlinear term}, Journal of Scientific Computing, 99 (2024), p.~26.

\bibitem{pfc}
{\sc K.R.~Elder and M.~Grant}, {\em Modeling elastic and plastic deformations in nonequilibrium processing using phase field crystals}, Physical Review E, 70 (2004), p.~051605.

\bibitem{ju2019}
{\sc X.~Li, L.~Ju, and X.~Meng}, {\em Convergence Analysis of Exponential Time Differencing Schemes for the Cahn--Hilliard equation}, Communications in Computational Physics, 26 (2019), p.~5.

\bibitem{ju2016}
{\sc L.~Zhu, L.~Ju, and W.~Zhao}, {\em Fast high-order compact exponential time differencing {Runge–Kutta} methods for second-order semilinear parabolic equations}, Journal of Scientific Computing, 67 (2016), pp.~1043--1065.

\bibitem{li2021}
{\sc J.~Li, L.~Ju, Y.~Cai, and X.~Feng}, {\em Unconditionally maximum bound principle preserving linear schemes for the conservative Allen--Cahn equation with nonlocal constraint}, Journal of Scientific Computing, 87 (2021), pp.~1--32.

\bibitem{lidongjsc}
{\sc D.~Li, and Z.~Qiao}, {\em On second order semi-implicit {Fourier} spectral methods for {2D Cahn--Hilliard} equation}, Journal of Scientific Computing, 70 (2017), pp.~301--341.

\bibitem{lidongsinum}
{\sc D.~Li, Z.~Qiao, and T.~Tang}, {\em Characterizing the stabilization size for semi-implicit {Fourier-spectral} method to phase field equations}, SIAM Journal on Numerical Analysis, 54 (2016), pp.~1653--1681.

\bibitem{caliari}
{\sc M.~Caliari, F.~Cassini, L.~Einkemmer, and A.~Ostermann}, {\em {Accelerating Exponential Integrators to Efficiently Solve Semilinear Advection-Diffusion-Reaction Equations}}, SIAM Journal on Scientific Computing, 46 (2024), pp.~A906--A928.

\bibitem{1997imex}
{\sc U.M.~Ascher, S.J.~Ruuth, and R.J.~Spiteri}, {\em Implicit-explicit {Runge--Kutta} methods for time-dependent partial differential equations}, Applied Numerical Mathematics, 25 (1997), pp.~151--167.

\bibitem{fuzhaohuiMC}
{\sc Z.~Fu, T~Tang, and J.~Yang}, {\em Energy diminishing implicit-explicit {Runge--Kutta} methods for gradient flows}, Mathematics of Computation, (2024).

\bibitem{Maset2008}
{\sc S.~Maset and M.~Zennaro}, {\em Unconditionally stability of explicit exponential {Runge--Kutta} methods for semi-linear ordinary differential equations}, Mathematics of Computation, 78 (2009), pp.~957--967.

\bibitem{chenwenbin1}
{\sc W.~Chen, W.~Li, Z.~Luo, C.~Wang, and X.~Wang}, {\em A stabilized second order exponential time differencing multistep method for thin film growth model without slope selection}, ESAIM: Mathematical Modeling and Numerical Analysis, 54 (2020), pp.~727--750. 

\bibitem{chenwenbin2}
{\sc W.~Chen, W.~Li, C.~Wang, S.~Wang, and X.~Wang}, {\em Energy stable higher-order linear {ETD} multi-step methods for gradient flows: application to thin film epitaxy}, Research in the Mathematical Sciences, 7 (2020), pp.~1--27.




\end{thebibliography}

\begin{appendices}
\appendix
\section{Proof of \Cref{proposition3.2}}
\label{secA}
It could be easily verified that 
\begin{equation}\nonumber
	\begin{aligned}
		&\left(-1+\lambda_{\ell,m}\right)^2+\kappa = 2 - 2\lambda_{\ell,m} + \lambda_{\ell,m}^2 + (\kappa-1) = \frac{2}{3} + \frac{4}{3} - 2\lambda_{\ell,m} + \frac{3}{4}\lambda_{\ell,m}^2 + \frac{1}{4}\lambda_{\ell,m}^2 + (\kappa-1)\\
		&=\frac{2}{3} + (\frac{2}{\sqrt{3}} - \frac{\sqrt{3}}{2}\lambda_{\ell,m})^2 + \frac{1}{4}\lambda_{\ell,m}^2+ (\kappa-1),
	\end{aligned}
\end{equation} 
so that
\begin{equation*}
	\Lambda_{\ell,m} = \left(-1+\lambda_{\ell,m}\right)^2+\kappa\ge \frac{1}{4}\lambda_{\ell,m}^2 + \frac{2}{3} + (\kappa-1) . 
\end{equation*}
This in turn leads to the following inequality: 
\begin{equation}\label{A3}
	\Vert G_i^{*} {f}\Vert_{2}^2\ge L^2\sum_{\ell,m=-K}^K\frac{1-\mathrm{e}^{-c_i\tau\Lambda_{\ell,m}}}{\tau\Lambda_{\ell,m}}[\frac{1}{4}\lambda_{\ell,m}^2 + (\kappa - 1) + \frac{2}{3}]\vert\hat{f}_{\ell,m}\vert^2.
\end{equation}
On the other hand, an application of Parseval's equality to the discrete Fourier expansion of $G_i^{\frac{1}{2}}\Delta_{N} {f}$ and $G_i^{\frac{1}{2}} {f}$ reveals that
\begin{align}
	\Vert G_i^{\frac{1}{2}}\Delta_{N} {f}\Vert_{2}^2 &= L^2\sum_{\ell,m=-K}^K\frac{1-\mathrm{e}^{-c_i\tau\Lambda_{\ell,m}}}{\tau\Lambda_{\ell,m}}\cdot\lambda_{\ell,m}^2\cdot\vert\hat{f}_{\ell,m}\vert^2\label{A4},\\
	\Vert G_i^{\frac{1}{2}} {f}\Vert_{2}^2 &= L^2\sum_{\ell,m=-K}^K\frac{1-\mathrm{e}^{-c_i\tau\Lambda_{\ell,m}}}{\tau\Lambda_{\ell,m}}\cdot\vert\hat{f}_{\ell,m}\vert^2\label{A5}.
\end{align}
Making a comparison between inequality \eqref{A3} and equations \eqref{A4}-\eqref{A5}, we conclude that inequality \eqref{eqn:proposition3.2(1)} has been established. Inequality \eqref{eqn:proposition3.2(2)} is a direct application of the above derivation process, and the details are skipped for the sake of brevity. This finishes the proof of \Cref{proposition3.2}.
\section{Proof of \Cref{proposition3.3}}
\label{secB}
The following expansion is assumed for the grid function $g$: 
\begin{equation}\nonumber
	g_{p,q} = \sum_{\ell,m=-K}^K\hat{g}_{\ell,m}\mathrm{e}^{2\pi\mathrm{i}(\ell x_p + my_q)/L}.
\end{equation}
In turn, the discrete Fourier expansion of $g-\varphi_0(c_1\tau L_{\kappa})f$ becomes
\begin{equation}\label{A1}
	(g - \varphi_0(c_i\tau L_{\kappa})f)_{p,q} = \sum_{\ell,m=-K}^K(\hat{g}_{\ell,m} - \mathrm{e}^{-c_i\tau\Lambda_{\ell,m}}\hat{f}_{\ell,m})\mathrm{e}^{2\pi\mathrm{i}(\ell x_p+my_q)/L},
\end{equation}
so that its discrete $\ell^2$ norm turns out to be
\begin{equation}\label{A2}
	\Vert g-\varphi_0(c_i\tau L_{\kappa})f \Vert_{2}^2 = L^2\sum_{\ell,m=-K}^K\vert \hat{g}_{\ell,m} - \mathrm{e}^{-c_i\tau\Lambda_{\ell,m}}\hat{f}_{\ell,m}\vert^2.
\end{equation}
Subsequently, a combination of equations \eqref{A1} and \eqref{A2} produces
\begin{equation}\nonumber
	\begin{aligned}
		& \tau\langle G_iL_{\kappa}f, \varphi_0(c_i\tau L_{\kappa})f\rangle  + \Vert g - \varphi_0(c_i\tau L_{\kappa})f\Vert_{2}^2 \\
		&= L^2\sum_{\ell,m=-K}^K\left((1-\mathrm{e}^{-c_i\tau\Lambda_{\ell,m}})\cdot\mathrm{e}^{c_i\tau\Lambda_{\ell,m}}\cdot\vert\mathrm{e}^{-c_i\tau\Lambda_{\ell,m}}\hat{f}_{\ell,m}\vert^2 + \vert \hat{g}_{\ell,m} - \mathrm{e}^{-c_i\tau\Lambda_{\ell,m}}\hat{f}_{\ell,m}\vert^2\right)\\
		&=L^2\sum_{\ell,m=-K}^K(1-\mathrm{e}^{-c_i\tau\Lambda_{\ell,m}})\left({\mathrm{e}^{c_i\tau\Lambda_{\ell,m}}\vert\mathrm{e}^{-c_i\tau\Lambda_{\ell,m}}\hat{f}_{\ell,m}\vert^2}\right.\\ &\left.{+ (1-\mathrm{e}^{-c_i\tau\Lambda_{\ell,m}})^{-1}\vert \hat{g}_{\ell,m} - \mathrm{e}^{-c_i\tau\Lambda_{\ell,m}}\hat{f}_{\ell,m}\vert^2}\right).\\
	\end{aligned}
\end{equation}
On the other hand, for each fixed mode frequency $(\ell,m)$, the following lower bound is clearly observed:
\begin{equation}\nonumber
	\begin{aligned}
		&\mathrm{e}^{c_i\tau\Lambda_{\ell,m}}a^2 + (1-\mathrm{e}^{-c_i\tau\Lambda_{\ell,m}})^{-1}b^2 = a^2+b^2 + (\mathrm{e}^{-c_i\tau\Lambda_{\ell,m}}-1)a^2 + [(1-\mathrm{e}^{-c_i\tau\Lambda_{\ell,m}})^{-1}-1]b^2\\
		&\ge a^2 + b^2 + 2ab = (a+b)^2,
	\end{aligned}
\end{equation}
for any $a,b\ge0$, in which the Cauchy inequality has been applied in the second step. Then we get 
\begin{equation}\nonumber
	\begin{aligned}
		& \mathrm{e}^{c_i\tau\Lambda_{\ell,m}}\vert\mathrm{e}^{-c_i\tau\Lambda_{\ell,m}}\hat{f}_{\ell,m}\vert^2 + 
		(1-\mathrm{e}^{-c_i\tau\Lambda_{\ell,m}})^{-1}\vert \hat{g}_{\ell,m} - \mathrm{e}^{-c_i\tau\Lambda_{\ell,m}}\hat{f}_{\ell,m}\vert^2\\
		&\ge(\vert\mathrm{e}^{-c_i\tau\Lambda_{\ell,m}}\hat{f}_{\ell,m}\vert + \vert\hat{g}_{\ell,m} - \mathrm{e}^{-c_i\tau\Lambda_{\ell,m}}\hat{f}_{\ell,m}\vert)^2\ge\vert \hat{g}_{\ell,m} \vert^2,
	\end{aligned}
\end{equation}
so that 
\begin{equation}\nonumber
	\begin{aligned}
		\tau\Vert G_iL_{\kappa}f,\varphi_0(c_i\tau L_{\kappa})f\Vert_{2} + \Vert g - \varphi_0(c_i\tau L_{\kappa})f\Vert_{2}^2\geq L^2\sum_{\ell,m=-K}^K(1-\mathrm{e}^{-c_i\tau\Lambda_{\ell,m}})\vert \hat{g}_{\ell,m}\vert^2
		=\tau\Vert G_i^{*}g\Vert_{2}^2.
	\end{aligned}
\end{equation}
Therefore, the proof of inequality \eqref{eqn:proposition3.3(1)} has been finished. The proofs of inequalities \eqref{eqn:proposition3.3(2)}, \eqref{eqn:proposition3.3(3)}, and \eqref{eqn:proposition3.3(4)} could be similarly derived as that of \eqref{eqn:proposition3.3(1)}, and the details are skipped to save space.
\section{Butcher-like tableaux of energy-stable methods}
We give the coefficients of the energy-stable methods that we used in the numerical experiments.
\begin{itemize}
	\item ETD1, also referred to as exponential Euler method:
	\begin{align}\label{eqn:etd1}
		\begin{array}
			{>{\centering\arraybackslash$} p{0.4cm} <{$} | >{\centering\arraybackslash$} p{0.6cm} <{$} }
			0 & 0 \\
			\hline
			1 & \varphi_1	
		\end{array}.
	\end{align}
	\item ETDRK2 reads as:
	\begin{align}\label{eqn:etdrk2}
		\begin{array}
			{>{\centering\arraybackslash$} p{0.4cm} <{$} | >{\centering\arraybackslash$} p{1.4cm} <{$} >{\centering\arraybackslash$} p{0.6cm} <{$}}
			0 & 0 & \\
			1 &  \varphi_{1} & 0\\
			\hline
			1 & \varphi_1-\varphi_{2} & \varphi_2	
		\end{array}.
	\end{align}
	\item  IMEX1, also referred to as first-order semi-implicit method:
	\begin{align}\label{eqn:imex1}
		\begin{array}
			{>{\centering\arraybackslash$} p{0.4cm} <{$} | >{\centering\arraybackslash$} p{0.5cm} <{$} >{\centering\arraybackslash$} p{0.5cm} <{$}}
			0 & 0 & 0 \\
			1 &  0 & 1\\
			\hline
			1 & 0 & 1	
		\end{array},\qquad
		\begin{array}
			{>{\centering\arraybackslash$} p{0.4cm} <{$} | >{\centering\arraybackslash$} p{0.5cm} <{$} >{\centering\arraybackslash$} p{0.5cm} <{$}}
			0 & 0 & 0 \\
			1 &  1 & 0 \\
			\hline
			1 & 1 & 0	
		\end{array}.
	\end{align}
	\item IMEX-RK(2,2) reads as (cf. \cite{1997imex}):
	\begin{align}\label{eqn:imex2}
		\begin{array}
			{>{\centering\arraybackslash$} p{0.4cm} <{$} | >{\centering\arraybackslash$} p{0.5cm} <{$} >{\centering\arraybackslash$} p{1.2cm} <{$} >{\centering\arraybackslash$} p{0.5cm} <{$}}
			0 & 0 & 0 & 0\\
			\gamma &  0 & \gamma & 0\\
			1 & 0 & 1-\gamma & \gamma\\
			\hline
			1 & 0 & 1-\gamma & \gamma	
		\end{array},\qquad
		\begin{array}
			{>{\centering\arraybackslash$} p{0.4cm} <{$} | >{\centering\arraybackslash$} p{0.5cm} <{$} >{\centering\arraybackslash$} p{1.2cm} <{$} >{\centering\arraybackslash$} p{0.5cm} <{$}}
			0 & 0 & 0  & 0 \\
			\gamma & \gamma & 0 & 0\\
			1 & \delta & 1-\delta & 0\\ 
			\hline
			1 & \delta & 1-\delta & 0	
		\end{array},
	\end{align}
	where $\gamma = \frac{2+\sqrt{2}}{2}$ and $\delta = \frac{2\gamma-1}{2\gamma}$. 
\end{itemize}
All the above schemes can preserve the \textbf{original} energy dissipation property for gradient flows including SH equation, see, e.g. \cite{LiX2024a, fuzhaohui,fuzhaohuiMC}.
\end{appendices}
\end{document}